\documentclass[a4paper]{amsart}
\usepackage{amssymb} 
\usepackage{amsmath} 
\usepackage{comment}
\usepackage{color}
\usepackage{array, booktabs, comment, mathrsfs, enumerate, url, multirow, mathabx}
\usepackage{tikz-cd, graphicx}
\usepackage{todonotes}
\usepackage{microtype}
\usepackage[neverdecrease]{paralist}
\usepackage[colorlinks]{hyperref}
\hypersetup{citecolor=blue}

\DeclareMathOperator{\Sym}{Sym}

\DeclareMathOperator{\PSL}{PSL}
\DeclareMathOperator{\AI}{AI}
\DeclareMathOperator{\Art}{Art}
\DeclareMathOperator{\Aut}{Aut}
\DeclareMathOperator{\BC}{BC}

\DeclareMathOperator{\Jac}{Jac}

\DeclareMathOperator{\N}{N}

\DeclareMathOperator{\Hom}{Hom}
\DeclareMathOperator{\GL}{GL}
\DeclareMathOperator{\SL}{SL}

\DeclareMathOperator{\Res}{Res}

\DeclareMathOperator{\Gal}{Gal}

\DeclareMathOperator{\End}{End}

\DeclareMathOperator{\SO}{SO}

\DeclareMathOperator{\GSp}{GSp}
\DeclareMathOperator{\Sp}{Sp}
\DeclareMathOperator{\GSO}{GSO}
\DeclareMathOperator{\GSpin}{GSpin}

\DeclareMathOperator{\Ind}{Ind}

\DeclareMathOperator{\Tr}{Tr}

\DeclareMathOperator{\rec}{rec}
\DeclareMathOperator{\WD}{WD}

\newcommand{\A}{{\mathbf A}}
\newcommand{\Q}{{\mathbf Q}}
\newcommand{\Z}{{\mathbf Z}}
\newcommand{\C}{{\mathbf C}}

\newcommand{\R}{{\mathbf R}}
\newcommand{\F}{{\mathbf F}}

\newcommand{\CO}{\mathcal{O}}

\newcommand{\gc}{{\mathfrak{c}}}

\newcommand{\gp}{\mathfrak{p}}
\newcommand{\gq}{\mathfrak{q}}
\newcommand{\gN}{\mathfrak{N}}

\newcommand{\Qbar}{\overline{\Q}}

\newcommand{\ua}{\underline{a}}

\newcommand{\uk}{\underline{k}}
\newcommand{\ul}{\underline{\lambda}}

\newcommand{\Frob}{\mathrm{Frob}}

\numberwithin{equation}{section}

\theoremstyle{plain}
\newtheorem{thm}[equation]{Theorem}
\newtheorem{lem}[equation]{Lemma}
\newtheorem{prop}[equation]{Proposition}

\newtheorem{cor}[equation]{Corollary}
\newtheorem{conj}[equation]{Conjecture}

\newtheorem{rem}[equation]{Remark}

\theoremstyle{definition}

\title[Lifts of Hilbert modular forms and applications]{Lifts of Hilbert modular forms and application to modularity of abelian varieties}

\begin{document}
\author{Clifton Cunningham}
\address[Cunningham]{Department of Mathematics and Statistics, University of Calgary,
2500 University Drive NW, Calgary, AB, Canada T2N 1N4}
\email{ccunning@ucalgary.ca}
\thanks{CC was supported by an NSERC grant}

\author{Lassina Demb\'el\'e}
\address[Demb\'el\'e]{Max-Planck Institute for Mathematics, Vivatsgasse 7, Bonn 53111, Germany}
\email{lassina.dembele@gmail.com}
\thanks{At the early stages of this project, LD was supported by EPSRC Grant EP/J002658/1, and later by a Visiting Scholar grant from the Max-Planck Institute for Mathematics}


\maketitle

\begin{quote}\it
To Benedict H. Gross on his recent retirement
\end{quote}

\begin{abstract} In this paper, we prove the existence of certain lifts of Hilbert cusp forms to
general odd spin groups. We then use those lifts to provide evidence for a conjecture of Gross on the modularity of
abelian varieties {\it not} of $\GL_2$-type.
\end{abstract}


\section{\bf Introduction}

Let $F/E$ be a cyclic extension of totally real number fields of degree $g$. Let $\pi$ be a cuspidal Hilbert automorphic representation of
$\GL_2(\A_F)$, with trivial central character, which is not a base change from any proper sub-extension $F'/E$ of $F$. In this paper, 
we show that the automorphic induction $\Pi'$ of $\pi$ from $F$ to $E$, which {\it a priori} lives on $\GL_{2g}(\A_E)$, 
descends to an automorphic representation $\Pi$ of $\GSpin_{2g+1}(\A_E)$. In fact, our result is a little stronger
(see Theorem~\ref{thm:theta-lifts}).

As an application of our construction, we provide some evidence for the following conjecture of Gross on the modularity of
abelian varieties (see Conjecture~\ref{conj:gross} for a slightly more general statement). 

\begin{conj}[Gross-Langlands]\label{conj:gross0}
\footnote{Although there was a folklore conjecture before it, to the best of our knowledge, this precise and verifiable formulation is due
to Gross. It is anologous to Weil's refinement of the Shimura-Taniyama conjecture for elliptic curves over $\Q$.}
Let $B$ be an abelian variety of dimension $g$ and conductor $N$ defined over $\Q$ such that $\End_\Q(B) = \Z$. Then, there exists a globally 
generic cuspidal automorphic representation $\Pi$ on $\GSpin_{2g + 1}(\A_\Q)$ of weight $2$ and paramodular level structure 
$N$, with field of rationality $\Q$, such that 
$$L(B, s) = L(\Pi, s).$$
\end{conj}

Gross' conjecture deals more generally with {\it discrete} symplectic motives of pure weight $1$, with Conjecture~\ref{conj:gross0} being a special
case, see \cite[Conjectures 3 and 4, and Proposition 6]{gross15}. It proposes a candidate admissible representation $\Pi = \bigotimes_v' \Pi_v$,
whose local components $\Pi_v$ are described using the local Langlands correspondence for $\SO_{2g+1}(\A_\Q)$ \cite{js04} after some appropriate
normalisation. It also predicts the level of $\Pi$ in terms of the conductor of $B$. The main challenge is to prove that this $\Pi$ is indeed automorphic and cuspidal. 
(We note that, although the field of rationality of $\Pi$ is not explicitly discussed in \cite{gross15}, its existence could be inferred from the theory of newforms developed in 
\S\S6-8. Also, the paramodular level structure is only defined for the group $\SO_{2g+1}$, but one can extend the theory to $\GSpin_{2g+1}$.)

Conjecture~\ref{conj:gross0} is a generalisation of the genus $2$ case also known as the pa\-ra\-mo\-du\-la\-ri\-ty conjecture \cite[Conjectures 1.1 and 1.4]{bk14};
the best theoretical evidence for the conjecture is only available in that case. Let $B$ be an abelian surface defined 
over $\Q$, with $\End_\Q(B) = \Z$. Let $N$ be the conductor of $B$, and $\ell\ge 5$ a prime not dividing $N$. Consider the Galois representation 
$$\rho_{B, \ell}:\, \Gal(\Qbar/\Q) \to \GSp_4(\Z_\ell),$$
arising from the $\ell$-adic Tate module of $B$. Under some technical conditions, Tilouine \cite{til06} and Pilloni \cite{pil12} prove 
an overconvergent ``$R = T$'' from which they deduce that there exists an overconvergent Siegel modular form $f$ of genus $2$ 
and weight $2$ such that $\rho_{B, \ell} \simeq \rho_{f, \ell}$. In this case, one is left with proving that the form 
$f$ is indeed classical. That crucial step has now been settled in \cite{pil17}. 
In \cite{cg16}, they also prove a modularity 
lifting theorem for Siegel modular forms of genus $2$ and non-regular weights. However, as it stands, this result doesn't apply yet to abelian surfaces. 
All of the results above impose a ``big image'' condition. This makes it extremely hard to check residual modularity in practice since $\GSp_4(\F_\ell)$ 
is rather big for $\ell\ge 5$. 


The infinite component $\Pi_\infty^{\rm h}$ of the automorphic representation $\Pi^{\rm h}$, associated to the form $f$ described above, 
is a limit of holomorphic discrete series. But, although the form $f$ itself is not cohomological, its construction relies crucially on the fact that
the symmetric space associated to $\SO(3,2)$ carries a complex structure. For example, in \cite{til06, pil12, pil17}, the form $f$ arises from
$\ell$-adic families of overconvergent Siegel modular forms. This requires working with the geometry of Siegel modular threefolds; this is
equally true for the approach used in \cite{cg16}.

Conjecture~\ref{conj:gross0} predicts an automorphic representation $\Pi$ which is globally generic. 
Its infinite component $\Pi_\infty$ is a limit of generic discrete series. 
Unfortunately, the symmetric space associated to the split group $\GSpin_{2g+1}$ no longer carries a 
complex structure when $g \ge 3$ (see \cite{gross15}). Moreover, as was pointed out to us by 
Tilouine \cite{til15},  $\Pi_\infty$ becomes more and more degenerate as  the dimension $g$ of $B$ increases. 
This means that, for $g\ge 3$, none of the approaches mentioned above readily applies, and that any 
substantial progress towards Conjecture~\ref{conj:gross0} will require some novel ideas.

Our goal in this paper is one that is more modest. We want to provide some evidence for Conjecture~\ref{conj:gross0} by means of functoriality. 
Our approach here is a generalisation of \cite{dk16, bdps15}. It consists in showing that Conjecture~\ref{conj:gross0}, or rather its converse, 
is compatible with Galois descent of isogeny classes, which corresponds to lifts on the automorphic side. More precisely, we prove the following:

\begin{thm}\label{thm:gross-evidence0} Let $F/\Q$ be a totally real cyclic extension of degree $g$, and write $\Gal(F/\Q) = \langle \sigma \rangle$.  
Let $\gN$ be an integral ideal such that $\gN^\sigma = \gN$, and $\pi$ a cuspidal Hilbert automorphic representation on $\GL_2(\A_F)$ 
of weight $2$ and level $\gN$, with trivial central character, which is not a base change from any proper subfield $F'$ of $F$. 
Let $L$ be the field of rationality of $\pi$. Assume that the Hecke orbit of $\pi$ is fixed by $\Gal(F/\Q)$. Further assume that there is an abelian variety
$A/F$ associated to $\pi$ by the Eichler-Shimura construction. 
Then, we have the followings:
\begin{enumerate}[(i)]
\item There exists a subfield $K$ of $L$ such that $L/K$ is cyclic of degree $g$;
\item The automorphic representation $\pi$ lifts to a globally generic cuspidal au\-to\-mor\-phic representation $\Pi$ of $\GSpin_{2g+1}(\A_\Q)$,
whose field of rationality is $K$;
\item There exists an abelian variety $B$ of dimension $g$ defined over $\Q$ such that $B\times_{\Q}F \sim A$,
$\End_{\Q}(B)\otimes\Q = K$, and
$$L(B, s) = \prod_{\Pi' \in [\Pi]}L(\Pi', s),$$ 
where $[\Pi]$ denotes the Hecke orbit of $\Pi$. 
\end{enumerate}
\end{thm}

\noindent
Combining Theorem~\ref{thm:gross-evidence0} with the Serre Conjecture \cite{ser87, kw09}, we obtain the following:

\begin{thm}\label{thm:prime-degree0} Let $B$ be an abelian variety of prime dimension $g$ over $\Q$ such that $\End_{\Q}(B) = \Z$. 
Suppose that there exist totally real number fields $F$ and $L$ of degree $g$ such that 
\begin{enumerate}[1.]
\item $F$ is cyclic;
\item $\End_F(A)\otimes \Q = L$, where $A = B \times_{\Q} F$;
\item There is a rational prime $\ell\ge 3$ and a (totally) ramified prime $\lambda \mid \ell$ in $L$ such that the
residual Galois representation $\bar{\rho}_{A,\lambda}:\,\Gal(\Qbar/F) \to \GL_2(\F_\ell)$ is surjective. 
\end{enumerate}
Then, $B$ is automorphic. More specifically, there exists a cuspidal Hilbert automorphic representation $\pi$
on $\GL_2(\A_F)$, which lifts to a globally generic cuspidal automorphic representation $\Pi$ on $\GSpin_{2g+1}(\A_\Q)$, 
such that $$L(B, s) = L(\Pi, s).$$
\end{thm}

The proof of Theorem~\ref{thm:prime-degree0} will show that Conditions (1) and (2) force $L$ to be cyclic as
a result of $B$ acquiring extra endomorphism after base change to $F$. Since $L/\Q$ is of prime degree, any
prime $\ell$ that ramifies in $L$ will be totally ramified. So, the technical requirement that we impose on $\bar{\rho}_{A,\lambda}$
is rather mild. It allows us to use some the most up to date modularity results available in
the literature, namely, the Serre Conjecture over $\Q$ \cite{ser87, kw09} and results in \cite{kt16, tho17} to show that there is an 
abundance of abelian varieties which verify Conjecture~\ref{conj:gross0}. We illustrate this by giving some examples. But, we also 
refer to \cite{dk16, bdps15} for additional examples of abelian surfaces. We believe that, by combining all modularity 
lifting results that are currently available, one should be able to remove Condition (3) from Theorem~\ref{thm:prime-degree0}. However, 
this would require a case by case detailed analysis that we do not want to perform here at the risk of overshadowing the primary
goal of this paper, which is to provide some evidence for Conjecture~\ref{conj:gross0}. We remark, in closing, that one
potential application of our results (see also Corollary~\ref{cor:gross-evidence1}), and their generalisation, is to the hypergeometric 
abelian varieties described in \cite{dar00} in connection with the Generalised Fermat Equations, see work of 
Billery-Chen-Dieulefait-Freitas \cite{bcdf17} on this problem.

The content of the paper is as follows. In Section~\ref{sec:auto-reps-gln}, we recall some background material on automorphic
representations on $\GL_n$ and their associated Galois representations. In Section~\ref{sec:lifts-hmf}, we prove the existence
of our lift. Section~\ref{sec:descent-ab-av} contains preliminary results on Galois descent of isogeny classes of abelian varieties;
they are results that could be interesting on their own. Section~\ref{sec:gross-conj} is dedicated to showing that our lift is compatible 
with Langlands functoriality; more precisely, we show that Conjecture~\ref{conj:gross0} is compatible with Galois descent of isogeny 
classes. Finally, we provide some illustrating examples in Sections~\ref{sec:examples} and~\ref{sec:ex-nf-gross}.

\subsection*{Acknowledgements} It is a pleasure to thank ICERM and Brown University for their hospitality, and the participants of the 
ICERM 2015 Semester Program on ``Computational Aspects of the Langlands Program'' where we first discussed this work. In particular, 
we would like to thank F.~Shahidi who suggested the strategy we use to construct our lifts, and A.~Brumer and R.~Schmidt for helpful 
discussions. We would also like to thank G.~Chenevier, L.~Dieulefait, N.~Freitas, T.~Gee, X.~Guitart, D.~Loeffler, S.~M\"uller, J.~Tilouine and 
J.~Voight for several helpful email exchanges and conversations. We are grateful to E.~Costa, N.~Mascot, J.~Sijsling and J.~Voight 
for allowing us to use their algorithm for computing endomorphism rings of abelian varieties, and for kindly providing us with
a preliminary version of their work. Finally, we would like to express our deepest gratitude to B.~Gross for his unbounded 
enthusiasm and encouragement during the course of this project.


\section{\bf Automorphic representations for $\GL_n$}\label{sec:auto-reps-gln}
In this section, we recall some results on automorphic representations for $\GL_n$ that we need. 
(We refer to \cite{bggt14a, bggt14b, bggt12} for more details.)
We let $F$ be a totally real number field, and denote by $J_F$ the set of all real embeddings of $F$. For every 
$\ua = (a_\tau)_{\tau \in J_F}$, with $a_\tau = (a_{\tau, 1},\ldots, a_{\tau, n}) \in \Z_{+}^{n}$, such that
$a_{\tau, 1} \ge \cdots \ge a_{\tau, n}$, we let $\Xi_{\ua}$ be the irreducible representation of $\Res_{F/\Q}(\GL_n)$ given by
$$\Xi_{\ua} = \bigotimes_{\tau \in J_F}\Xi_{a_\tau},$$
where $\Xi_{a_\tau}$ is the irreducible representation of $\GL_n$ of highest weight $a_\tau$. We are interested in
the $\ua$ such that $\Xi_{\ua}$ is an algebraic representation of  $\Res_{F/\Q}(\GL_n)$. In that case, there 
exists $n_0 \in \Z_{+}$ such that, for all $1\le i \le n$, 
$$a_{\tau, i} + a_{\tau, n + 1 - i} = n_0.$$

Let $\pi$ be an automorphic representation of $\GL_n(\A_F)$. We say that $\pi$ has weight $\ua$ if $\pi_\infty$
has the same infinitesimal character as $\Xi_{\ua}^\vee$, the dual of $\Xi_{\ua}$. We will be interested in automorphic
representations $\pi$ which satisfy the following conditions:
\begin{enumerate}[1.]
\item $\pi$ is {\it regular algebraic} of weight $\ua$, i.e. the weight $\Xi_{\ua}$ is an algebraic representation of $\Res_{F/\Q}(\GL_n)$;
\item $\pi$ is {\it essentially self dual}, i.e. $\pi^\vee \simeq \pi \otimes \chi$ for some Hecke character 
$\chi:\,F^\times\backslash\A_F^\times \to \C^\times$ such that $\chi_\tau(-1) = (-1)^n$ for all $\tau \in J_F$;
\item $\pi$ is {\it cuspidal}. 
\end{enumerate}
An automorphic representation $\pi$ which satisfies the conditions above will be called {\it RAESDC} (regular, algebraic, essentially self-dual and
cuspidal). We will also use the same terminology for the pair $(\pi, \chi)$. 

\begin{rem}\rm Since strong multiplicity one holds for $\GL_n$, Condition (2) is equivalent to saying that for almost all unramified places 
$v$ we have $\pi_v \simeq \pi_v^\vee$.
\end{rem}

\subsection{ Field of rationality and Hecke orbits} Let $\pi$ be an RAESDC automorphic representation of $\GL_n(\A_F)$, 
and write $\pi = \pi_\infty \otimes \pi_f$, where $\pi_\infty$ is the archmedian component of $\pi$, and $\pi_f$ is finite part.
Let $V$ be the underlying complex vector space of $\pi_f$. For $\tau \in \Aut(\C)$, let $V^{\tau} = V \otimes_{\C, \tau}\C$.
Then $V^{\tau}$ admits a natural admissible representation of $\GL_n(\A_F^f)$, which we denote by $\pi_f^\tau$. So, we
get an action of $\Aut(\C)$ on equivalence classes of admissible representations of $\GL_n(\A_F^f)$. Let $H(\pi_f)$ be
the stabiliser of the class of $\pi_f$ in $\Aut(\C)$, so
$$H(\pi_f) =\left\{\tau \in \Aut(\C): \pi_f^\tau \simeq \pi_f\right\}.$$
The {\it field of rationality} of $\pi$ denoted $L_{\pi}$ is defined by
$$L_{\pi} := \C^{H(\pi_f)}.$$
In general, $L_{\pi}$ is not a number field. But in our step up, this will always be the case thanks to
 \cite[Proposition 3.1]{clo90} and \cite[Theorem 1.11]{clo14}. 

\begin{thm}[Clozel]\label{thm:clozel-fld-rationality} Let $\pi$ be an RAESDC automorphic representation on $\GL_n(\A_F)$. Then, the
field of rationality $L_{\pi}$ is a totally real or {\rm CM} number field. 
\end{thm}
Let $\pi$ be an RAESDC automorphic representation of $\GL_n(\A_F)$. Then the admissible representations $\pi_f$ of $\GL_n(\A_F^f)$ 
has a model $\pi_{f,0}$ over $L_{\pi}$. This model is unique up to scaling, and is stable under 
$\GL_n(\A_F^f)$ (see {\it loc. cit.}). We define the {\it Hecke orbit} of $\pi$ to be the set
\begin{align*}
[\pi] &:= \left\{\pi^\tau =\pi_\infty \otimes \pi_f^\tau : \tau \in \Hom(L_\pi, \C) \right\}.
\end{align*}
We note that, for $\tau \in \Hom(L_\pi, \C)$, the field of rationality of $\pi^\tau$ is $L_{\pi^\tau} := \tau(L_\pi)$, and
that $\pi^\tau_{f,0} = \pi_{f,0}\otimes_{L_\pi, \tau} L_{\pi^\tau}$ is a model for $\pi_f^\tau$.

\subsection{Galois representations} The following result is the culmination of the work of lots of people 
(see \cite[Theorem 2.1.1]{bggt14b} and reference therein.)

\begin{thm}\label{thm:galois-reps-for-gln} Let $(\pi,\chi)$ be an RAESDC automorphic representation of weight $\ua$ 
on $\GL_n(\A_F)$. Let $\lambda$ be a prime of $L_{\pi}$ with residue characteristic $\ell$, and fix an isomorphism 
$\iota:\,\overline{L}_{\pi,\lambda} \simeq \C$. Then, there exists a Galois representation 
$$r_\lambda(\pi): G_F \to \GL_n(\overline{L}_{\pi,\lambda}),$$ such that 
\begin{enumerate}[(i)]
\item $r_{\lambda}(\pi)^\vee \simeq r_{\lambda}(\pi)r_{\lambda}(\chi)$;
\item $\WD(r_\lambda(\pi)|_{G_{F_v}})^{\text{\rm F-ss}} \simeq \rec_{F_v}(\pi_v\otimes |\!\det\!|^{(1-n)/2})$;
\item For each $v\mid \ell$, $r_\lambda(\pi)|_{G_{F_v}}$ is de Rham, and for each embedding $\tau:\,F \hookrightarrow \overline{L}_{\pi,\lambda}$, 
the Hodge-Tate weights are given by
$$\mathrm{HT}_\tau(r_\lambda(\pi)) = \big\{a_{\tau, n}, a_{\tau, n - 1} + 1,\, \ldots,\,a_{\tau, 1} + n - 1\big\}.$$
\end{enumerate}
These conditions determine $r_{\lambda}(\pi)$ uniquely up to isomorphism. 
\end{thm}
The requirement that $\pi$ be essentially self-dual in Theorem~\ref{thm:galois-reps-for-gln} has been removed in recent work of Harris-Lan-Taylor-Thorne \cite[Theorem A]{hltt16}
(also see \cite[Theorem 5.1.4]{sch15}), but we will not need this here. However, we will be dealing with automorphic representations that
are essentially self-dual and cuspidal but {\it not} regular (AESDC), and for which Theorem~\ref{thm:galois-reps-for-gln} will still be true. So, for the
remainder of this section, we will assume that the Langlands correspondence is known for all our automorphic representations. 

\subsection{Base change for $\GL_n$} We recall that, if $F/E$ is a Galois extension, then there is a natural action of $\Gal(F/E)$ 
on the set of automorphic representations on $\GL_n(\A_F)$ given by
$${}^\sigma\!\pi := \pi \circ \sigma.$$
If $\pi$ has weight $\ua$, then ${}^\sigma\!\pi$ has weight $\sigma\cdot\ua := (a_{\sigma\tau})$, $\sigma \in \Gal(F/E)$, with 
$$\Xi_{\sigma\cdot\ua} = \bigotimes_{\tau \in J_F}\Xi_{a_{\sigma\tau}}.$$ 
Solvable base change for $\GL_n$ for cyclic extensions of prime degree is due to Arthur-Clozel \cite{ac89}. Their results were 
extended to general cyclic extensions by Henniart \cite{hen12}. 

\begin{thm}[Cyclic base change]\label{thm:base-change}
Let $F/E$ be a cyclic extension of totally real fields. Let $\pi$ be a cuspidal automorphic representation of $\GL_n(\A_E)$ of weight $\ua$. 
Then there exists a cuspidal automorphic representation $\Pi$ of $\GL_n(\A_F)$ such that
\begin{enumerate}[(i)]
\item For all finite places $v$ of $F$, $\rec_{F_v}(\Pi_v) = \rec_{E_{w}}(\pi_w)|_{W_{F_v}}$, where $v \mid w$.
In particular, we have $r_\lambda(\Pi) \simeq r_\lambda(\pi)|_{G_F}$. 
\item $\Pi \simeq \Pi\circ \sigma$, where $\Gal(F/E) =\langle \sigma \rangle$. 
\end{enumerate}
The automorphic representation $\Pi$ is uniquely determined by those two conditions. It is called the {\rm base change} of $\pi$
from $E$ to $F$, and denoted by $\BC_{F/E}(\pi)$. It satisfies the additional property:
\begin{enumerate}
\item[(iii)] $\BC_{F/E}(\pi) \simeq \BC_{F/E}(\pi')$ if and only if $\pi \simeq \pi'\otimes (\delta_{F/E}^i\circ \Art_E\circ \det)$ for some $i$,
where $\Gal(F/E)^\vee = \langle \delta_{F/E}\rangle$ is the dual of $\Gal(F/E)$.
\item[(iv)] The weight $\BC_{F/E}(\ua)$ of $\Pi$ is given by $\BC_{F/E}(\ua)_{\tau} = a_{\tau'}$ where $\tau' = \tau|_{E}$.
\end{enumerate}
\end{thm}

\subsection{Automorphic induction} In {\it loc. cit.}, Arthur-Clozel also proved automorphic induction for cyclic extensions $F/E$ of prime degree.
This result was also extended to general cyclic extensions in \cite{hen12}.

\begin{thm}[Automorphic induction]\label{thm:auto-induction} Let $F/E$ be a cyclic extension of totally number fields of degree $g$, and $\pi$ 
a cuspidal automorphic representation of $\GL_m(\A_F)$, induced from cuspidal, such that $\pi$ is not a base change from any sub-extension $F'/E$ of $F$. Then, there exists 
a cuspidal automorphic representation $\Pi$ of $\GL_{mg}(\A_E)$, induced from cuspidal, which satisfies the following properties:
\begin{enumerate}[(i)]
\item $\BC_{F/E}(\Pi) = \pi \boxplus \pi^\sigma \boxplus \cdots \boxplus \pi^{\sigma^{g-1}}$;
\item $\Pi \simeq \Pi \otimes (\delta_{F/E}\circ \Art_E\circ \det)$. 
\end{enumerate}
The automorphic representation $\Pi$, which is uniquely determined by Conditions (i) and (ii) up to isomorphism, is called 
the {\rm automorphic induction} of $\pi$, and denoted by $\AI_{E/F}(\pi)$. It satisfies the additional property:
\begin{enumerate}
\item[(iii)] For all finite places $w$ of $E$, we have $\rec_{w}(\AI_{E/F}(\pi)_w) = \rec_{w}(\AI_{E_w/F_v}(\pi_v))$, where $w \mid v$.
\end{enumerate}
\end{thm}

\begin{rem}\rm We have stated Theorem~\ref{thm:auto-induction} in such a way that the characterising properties
of automorphic induction mirror those of base change. Not only is this very useful from a practical point of view, but
also these functorial properties guide the proof in \cite{hen12}, where automorphic  induction is constructed by 
combining base change with Galois descent.
\end{rem}

\subsection{Field of rationality of automorphic induction} 
It will be important for us to know the field of rationality of the automorphic representations we obtain from induction. 
So, we will need the following lemma.

\begin{lem}\label{lem:fld-of-rationality} Let $\pi$ be a RAESDC automorphic representation of $\GL_m(\A_F)$ such that $\pi$ is not a base change from
any proper sub-extension $F'/E$ of $F$. Let $L_{\pi}$ and $L_{\Pi}$ be the fields of rationality of $\pi$ and $\Pi = \AI_{E/F}(\pi)$, respectively. Then,
$L_{\Pi}$ is a subfield of $L_\pi$. If further $[\pi] = [{}^\sigma\!\pi]$ (i.e. the Hecke orbit of $\pi$ is preserved 
by $\Gal(F/E)$), then $L_{\pi}$ is a cyclic extension of $L_{\Pi}$ of degree $g$. 
\end{lem}

\begin{proof} By Theorem~\ref{thm:clozel-fld-rationality}, $L_\pi$ is a number field, which is either totally real or CM, and by 
definition of base change, we have  $L_{\pi} = L_{{}^\sigma\!\pi}$. From Theorem~\ref{thm:auto-induction} (1), we see 
that $L_{\Pi}$ is the smallest  field over which the cuspidal support of $\Pi$ is defined. Since all the $\Gal(F/E)$-conjugate 
of $\pi$ are defined over $L_\pi$, we conclude that $L_{\Pi}$ must be a subfield of $L_{\pi}$. This gives the first assertion. To prove the second one, observe that
if $[\pi] = [{}^\sigma\!\pi]$, then there exists $\tau \in \Hom(L_{\pi}, \C)$ such that ${}^\sigma\!\pi = \pi^\tau.$ By uniqueness of models, this implies that 
${}^\sigma\!\pi_{f,0} = \pi_{f,0}^\tau$. Hence, $L_{\pi} = L_{\pi^\tau}$ and $\tau \in \Aut(L_\pi)$. The map $\sigma \mapsto \tau$ must be an injection. 
Otherwise,  this would contradict the fact that the stabiliser of $\pi$ in $\Gal(F/E)$ is trivial. Let $H$ be the subgroup
of $\Aut(L_\pi)$ generated by $\tau$ and $K = L_\pi^{H}$. Then, $\Pi$ is clearly defined over $K$, and it is not hard to see that $K$ is the smallest such field.
So $L_\Pi = K$. 
\end{proof}

\begin{rem}\rm In general, it is a very delicate matter to determine the field of rationality of an automorphic representation $\Pi$
whose infinite components $\Pi_v$, $v \mid \infty$, are limit of discrete series (see \cite{bhr94, st14} for example) as is the case in our applications to
abelian varieties. However, since our representations are automorphic inductions of representations whose weights are regular, 
Lemma~\ref{lem:fld-of-rationality} ensures that their fields of rationality are still totally real or CM. 
\end{rem}

\begin{lem}\label{lem:self-dual} Let $(\pi,\chi)$ be an essentially self-dual cuspidal automorphic representation on $\GL_{m}(\A_F)$.
Then $(\AI_{E/F}(\pi), \chi_{|E})$ is an essentially self-dual cuspidal automorphic representation on $\GL_{mg}(\A_E)$. 
\end{lem}

\begin{proof} Let $\Pi = \AI_{E/F}(\pi)$, so that $\Pi^\vee = \AI_{E/F}(\pi^\vee)$. Then, by Theorem~\ref{thm:auto-induction}, $\Pi$ is cuspidal. So, we only need to show
that it is essentially self-dual. We recall that $\pi^\vee \simeq \pi \otimes \chi$, and for simplicity, we assume that the stabiliser 
of $\pi$ in $\Gal(F/E)$ is trivial. By Theorems~\ref{thm:base-change} and~\ref{thm:auto-induction}, it follows that
\begin{align*}
\BC_{F/E}(\Pi^\vee) &= \pi^\vee \boxplus (\pi^\vee)^\sigma \boxplus \cdots \boxplus (\pi^\vee)^{\sigma^{g-1}}\\
&= (\pi \otimes \chi) \boxplus (\pi \otimes \chi)^\sigma \boxplus \cdots \boxplus (\pi \otimes \chi)^{\sigma^{g-1}}\\
&= (\pi \boxplus \pi^\sigma \boxplus \cdots \boxplus \pi^{\sigma^{g-1}}) \otimes \chi_{|E}\circ\mathrm{N}_{F/E}\\
&= \BC_{F/E}(\Pi) \otimes \BC_{F/E}(\chi_{|E})\\
&= \BC_{F/E}(\Pi \otimes \chi_{|E}). 
\end{align*}
So, by uniqueness, we must have $\Pi^\vee = \Pi \otimes \chi_{|E}$. This implies that $(\Pi, \chi_{|E})$ is essentially self-dual. 
\end{proof}

\subsection{\bf Automorphic descent from $\GL_{2g}$ to $\GSpin_{2g+1}$}\label{sec:auto-descent}
Let $G/F$ be the algebraic group $\GSpin_{2g+1}$ or the split $\GSpin_{2g}$ or one of its non-split
quasi-split forms $\GSpin_{2g}^*$ determined by some quadratic extension of $F$. 
(For a detailed description of these groups, we refer to \cite{as06, as14}.) The dual group of
$G$ is $\widehat{G} = \GSp_{2g}$ or $\GSO_{2g}$, accordingly. 
There is a natural embedding $\widehat{G} \hookrightarrow \GL_{2g}$, which gives rise to a homomorphism 
$$\iota:\,\widehat{G}(\C) \hookrightarrow \GL_{2g}(\C).$$
Let $(\Pi',\chi)$ be an AESDC automorphic representation of $\GL_{2g}(\A_F)$. 
Let $\Pi$ be a cuspidal automorphic representation of $G(\A_F)$. Then, for each place $v$ of $F$, the
local component $\Pi_v$ of $\Pi$ is determined by its local Langlands parameter 
$$\phi_v: W_{F_v} \to \widehat{G}(\C).$$
We say that $(\Pi',\chi)$ is the {\it weak functorial transfer} of $\Pi$ to $\GL_{2g}$ if, for almost place $v$ of $F$,
the Langlands parameter of the local component $\Pi_v'$ is given by $\iota \circ \phi_v$. 
We say that it is the {\it strong functorial transfer} if this is true for all $v$. 

Let $(\Pi',\chi)$ be as above. Then, for every finite set of place $S$ (containing the archimedians), 
we have the equality of incomplete $L$-series
$$L^S(s, \Pi' \times (\Pi' \otimes \chi^{-1})) = L^S(s, \Sym^2(\Pi') \otimes \chi^{-1}) L^S(s, (\wedge^2\Pi') \otimes \chi^{-1}).$$
For $S$ sufficiently large, the left hand side has a simple pole at $s = 1$ (see \cite{js81a, js81b}). Hence, only one of the $L$-series
of the right hand side must have a pole. The following theorem characterises the image of automorphic descent to $\GSpin_{2g+1}$. 
(We refer to \cite{sha11} for the notion of global genericity, and to their Conjecture 2.10, which predicts that every cuspidal 
automorphic representations of $\GSpin_{2g+1}$ that is everywhere local generic is globally generic. We also refer 
to \cite{as06, as14, hs12} for a detailed study of the question of functorial transfer from $\GSpin_{2g+1}$ to $\GL_{2g}$.)

\begin{thm}\label{thm:descent} Let $F$ be a number field, and $(\Pi',\chi)$ an essentially self-dual cuspidal automorphic representation of $\GL_{2g}(\A_F)$. 
Then, the followings are equivalent:
\begin{enumerate}[(a)]
\item There exists a globally generic, cuspidal, automorphic representation $\Pi$ of $\GSpin_{2g+1}(\A_F)$, whose unique functorial transfer 
to $\GL_{2g}(\A_F)$ is $(\Pi',\chi)$;
\item The $L$-series $\displaystyle{L^S(s, \wedge^2\Pi'\otimes \chi^{-1})}$ has a pole at $s = 1$.
\end{enumerate}
\end{thm}
\begin{proof} Asgari-Shahidi \cite[Theorem 4.26]{as14} prove that (a) implies (b), while Hundley-Sayag \cite[Corollary 3.2.1]{hs12}
show the converse.
\end{proof}

\begin{rem}\rm It is important to note that \cite[Theorem 4.26]{as14} only gives that $\Pi'$ is a weak transfer of $\Pi$. However, \cite[Proposition 5.1]{as14}
shows that the transfer is in fact strong.
\end{rem}

\section{\bf Lifts of Hilbert automorphic representations}\label{sec:lifts-hmf}
In this section, we prove the existence of lifts of Hilbert automorphic forms. 

\subsection{Automorphic representations on $\GL_2$}
Let $k\ge 2$ and $w$ be two integers with the same parity. We let $\mathrm{D}_{k, w}$ be
the {\it discrete} series representation of $\GL_2(\R)$ defined in \cite[\S 0.2]{carayol86};
it is the essentially square integrable sub-representation of the unitary induction $\Ind(\mu, \nu)$, 
where $\mu,\nu$ are the characters of $\R^\times$ given by
$$\mu(t) = |t|^{\frac{k-1-w}{2}}\cdot \mathrm{sgn}(t),\,\,\, \nu(t) = |t|^{\frac{-k+1-w}{2}}.$$
The central character of $\mathrm{D}_{k,w}$ is the character $\mu\nu: t \mapsto t^{-w}$. 

Let $\uk = (k_\tau)_{\tau \in J_F} \in \Z_{+}^{J_F}$ be such that the $k_\tau$ have the same parity. Let 
$k_0 = \max\left\{k_\tau:\, \tau \in J_F\right\}$ and $w = k_0 - 2$. For each $\tau \in J_F$, 
define $m_\tau = \frac{k_0-k_\tau}{2}$ and $n_\tau = k_\tau - 2$. We are interested in automorphic representations 
$\pi =\bigotimes_v \pi_v$ of $\GL_2(\A_F)$ such that  $\pi_{\tau} \simeq \mathrm{D}_{k_\tau, w}$. Such 
representations correspond to holomorphic Hilbert modular forms of weight $\uk$. In the terminology of 
Section~\ref{sec:auto-reps-gln}, we see that $\pi$ has weight $\ua = (a_\tau)_{\tau \in J_F}$, where
$a_\tau = (a_{\tau, 1}, a_{\tau, 2}) = (k_\tau + m_\tau - 2, m_\tau)$. (Note that $a_{\tau, 1} + a_{\tau, 2} = w$.)

\subsection{Galois representations} 
Let $\pi$ be a cuspidal Hilbert automorphic representation of $\GL_2(\A_F)$ of weight $\uk$, and $L_\pi$ the field of rationality of $\pi$. Then, 
by work of Shimura \cite{shimura78}, $L_{\pi}$ is the coefficient field of the newform $f \in \pi$, i.e. $L_\pi = \Q(\{a_\gp(f):\,\gp \subset \CO_F\})$ 
where $a_\gp(f)$ is the eigenvalue of the Hecke $T_\gp$ at $\gp$ acting on $f$. In this case, 
Theorem~\ref{thm:galois-reps-for-gln} is a bit more precise. In the form below, it follows from work of Carayol \cite{carayol86}, 
Taylor \cite{tay89} and many other people.

\begin{thm}\label{thm:galois-reps-for-hmf} Let $\ell$ be a rational primes, and $\lambda \mid \ell$ a prime in $L_\pi$. 
Then, there exists a Galois representation 
$$r_\lambda(\pi): G_F \to \GL_2(\overline{L}_{\pi, \lambda}),$$ such that 
\begin{enumerate}[1.]
\item For each finite place $v$ such that $\pi_v$ is unramified, and $v \nmid \ell$, then the characteristic polynomial of $\Frob_v$ is given by
$X^2 - a_v X + q_v s_v$, where $q_v = \#k_v$, $k_v$ the residue field of $\CO_{F_v}$ at $v$, and $a_v$ and $s_v$ are the eigenvalues of the Hecke operators
$T_v$ and $S_v$ respectively. 

\item More generally, we have $\WD(r_\lambda(\pi)|_{G_{F_v}})^{\text{\rm F-ss}} \simeq \rec_{F_v}(\pi_v\otimes |\!\det\!|^{-1/2})$. 
\item If $v \mid \ell$, then $r_\lambda(\pi)|_{G_{F_v}}$ is de Rham with Hodge-Tate weights 
$$\mathrm{HT}_\tau(r_\lambda(\pi))= \{m_\tau, k_\tau + m_\tau - 1\},$$
for each embedding $\tau: F \hookrightarrow \overline{L}_{\pi,\lambda} \simeq \C$.
If $\pi_v$ is unramified then $r_\lambda(\pi)|_{G_{F_v}}$ is crystalline. 
\item If $c_v$ is a complex conjugation, then $\det(r_\lambda(\pi)(c_v)) = -1$. 
\end{enumerate}
\end{thm}

\subsection{\bf Local-global compatibility at archimedian places}\label{subsec:lg-comp}
Conditions (2) and (3) express local-global compatibility conditions at finite places. This also extends to 
archimedian places. Let $\zeta_{k_\tau, w}$ be the character
\begin{align*}
\zeta_{k_\tau, w}: \C^\times &\to \C^\times\\
 z &\mapsto (z\bar{z})^{\frac{-w-k_\tau+1}{2}}z^{k_\tau-1}.
 \end{align*}
This character determines the $2$-dimensional Weil representaion
$\sigma_{k_\tau, w} = \Ind_{W_{\C}}^{W_\R}\zeta_{k_\tau, w}$. 
The local-global compatibility condition at infinite places, asserts that $\sigma_{k_\tau, w}$ corresponds to $\mathrm{D}_{k_\tau, w}$.

\subsection{Lifts of Hilbert automorphic representations}
We are now ready to prove our main result on the existence of lifts. The method we used was suggested by F.~Shahidi.
It follows somewhat \cite{rs07}, with the advantage that we now have the results in \cite{as06, as14} and \cite{hs12} at our disposal in the form of
Theorem~\ref{thm:descent}. 

\begin{lem}\label{lem:cusp-decomp} Let $\pi$ be a cuspidal automorphic representation of $\GL_2(\A_F)$,
with central character $\omega$, which is not a base change 
from any proper sub-extension $F'/E$ of $F$. Let $\Pi = \AI_{E/F}(\pi)$, and $t = \lfloor\frac{g-1}{2}\rfloor$. 
When $g$ is even, let $E_0/E$ be the sub-extension of degree $g/2$ of $F$, and $\Pi_0$ a cuspidal automorphic 
representation over $\GL_4(\A_{E_0})$ such that $\BC_{F/E_0}(\Pi_0) = \pi \times \pi^{\sigma^{g/2}}$.
Then we have the following isobaric decompositions:
\begin{align*}
\wedge^2 \Pi &= \AI_{E/F}(\omega)|\cdot|^{k_0-1}\boxplus \bigboxplus_{i=1}^t \AI_{E/F}(\pi \times \pi^{\sigma^i}) \boxplus \left\{
\begin{array}{ll}
0&\text{if}\,\,g\,\,\text{is odd};\\ \\
\AI_{E/E_0}(\Pi_0)&\text{otherwise};
\end{array}\right.\\ \\
\Sym^2 \Pi &= \AI_{E/F}(\Sym^2(\pi))\boxplus \bigboxplus_{i=1}^t \AI_{E/F}(\pi \times \pi^{\sigma^i}) \boxplus \left\{
\begin{array}{ll}
0&\text{if}\,\,g\,\,\text{is odd};\\ \\
\AI_{E/E_0}(\Pi_0)&\text{otherwise}. 
\end{array}\right.
\end{align*} 

\end{lem}

\begin{proof} Since $F$ is cyclic, and $\pi$ is not a base change from any proper sub-extension $F'/E$ of $F$, 
Theorem~\ref{thm:auto-induction} implies that $\Pi = \AI_{E/F}(\pi)$ is a cuspidal automorphic representation of $\GL_{2g}(\A_E)$.
Recall that $\pi^\vee = \pi \otimes \omega^{-1}$, and $\det(r_{\lambda}(\pi)) = r_{\lambda}(\omega) \epsilon^{k_0-1}$. 
So, by Lemma~\ref{lem:self-dual}, it follows $\Pi^\vee= \Pi \otimes \omega_{|E}^{-1}$.
By Theorem~\ref{thm:auto-induction}, we have 
\begin{align*}
\left(\wedge^2 r_\lambda(\Pi)\right)\Big|_{G_F} &=  \bigwedge^2\left(r_\lambda(\pi) \oplus \cdots \oplus r_\lambda(\pi^{\sigma^{g-1}})\right)\\
&=\left(r_{\lambda}(\omega) \oplus r_{\lambda}(\omega^{\sigma}) \oplus \cdots \oplus r_{\lambda}(\omega^{\sigma^{g-1}})\right)\otimes\epsilon^{k_0-1}\\
&\qquad {} \oplus \bigoplus_{0\le i<j\le g-1} r_\lambda(\pi^{\sigma^{i}}) \otimes r_\lambda(\pi^{\sigma^{j}})\\
&=\left(\Ind_E^F(r_{\lambda}(\omega))\big|_{G_F}\right)\otimes\epsilon^{k_0-1} \oplus \bigoplus_{0\le i<j\le g-1} r_\lambda(\pi^{\sigma^{i}} \times \pi^{\sigma^{j}}).
\end{align*}
A similar calculation shows that 
\begin{align*}
\left(\Sym^2 r_\lambda(\Pi)\right)\Big|_{G_F} &=\left(\Ind_E^F(\Sym^2(r_{\lambda}(\pi)))\big|_{G_F}\right) \oplus \bigoplus_{0\le i<j\le g-1} r_\lambda(\pi^{\sigma^{i}} \times \pi^{\sigma^{j}}).
\end{align*}
When $g$ is odd, then we can rearrange the last sum as
\begin{align*}
\bigoplus_{0\le i<j\le g-1} r_\lambda(\pi^{\sigma^{i}} \times \pi^{\sigma^{j}}) &= \bigoplus_{i = 1}^t \left(r_\lambda(\pi \times \pi^{\sigma^{i}}) \oplus \cdots 
\oplus r_\lambda(\pi^{\sigma^i} \times \pi^{\sigma^{i+g-1}})\right)\\
&=\bigoplus_{i=1}^t \left(\Ind_E^F(r_\lambda(\pi\times \pi^{\sigma^i}))\big|_{G_F}\right).
\end{align*}
When $g$ is even, then $\pi \times \pi^{\sigma^{g/2}}$ is a base change from the subfield $E_0$ of $F$ of degree $g/2$. (It is cuspidal, see below.) 
Let $\Pi_0$ be such that $\BC_{F/E_0}(\Pi_0) = \pi \times \pi^{\sigma^{g/2}}$. Then, $\Pi_0$ is cuspidal by Theorem~\ref{thm:base-change},
and it is not hard to see that it cannot be a base change from any proper subfield $E'/E$ of $E_0$. We have
\begin{align*}
\bigoplus_{0\le i<j\le g-1} r_\lambda(\pi^{\sigma^{i}} \times \pi^{\sigma^{j}}) &= \bigoplus_{i = 1}^t \left(r_\lambda(\pi \times \pi^{\sigma^{i}}) \oplus \cdots 
\oplus r_\lambda(\pi^{\sigma^i} \times \pi^{\sigma^{i+g-1}})\right)\\
&\qquad \oplus \left(r_{\lambda}(\Pi_0)\big|_{G_F} \oplus \cdots \oplus r_{\lambda}(\Pi_0^{\sigma^{g/2}})\big|_{G_F}\right)\\
&=\bigoplus_{i=1}^t \left(\Ind_E^F(r_\lambda(\pi\times \pi^{\sigma^i}))\big|_{G_F} \right) \oplus \Ind_{E}^{E_0}(r_{\lambda}(\Pi_0))\big|_{G_F}.
\end{align*}
By \cite{js81a, js81b}, $\pi^{\sigma^{i}} \times \pi^{\sigma^{j}}$ is cuspidal, for $0\le i< j \le g-1$.
So, we conclude the lemma by again applying Theorem~\ref{thm:auto-induction}. 
\end{proof}

\begin{thm}\label{thm:theta-lifts} 
Let $F/E$ be a cyclic extension of totally real fields of degree $g$. Let $\pi$ be a cuspidal Hilbert automorphic representation, with central character $\omega$, 
which is not a base change from any sub-extension $F'/E$ of $F$. Then, $\Pi'= \AI_{E/F}(\pi)$ descends to a globally generic, cuspidal, automorphic representation 
$\Pi$ of $\GSpin_{2g+1}(\A_E)$ if and only if $\omega$ factors through the norm map $\mathrm{N}_{F/E}: F \to E$.
\end{thm}

\begin{proof} 
Let $r = \lfloor \frac{g-1}{2}\rfloor$ or $g/2$ according as to $g$ is odd or even. Then, by Lemma~\ref{lem:cusp-decomp}, there exist cuspidal 
automorphic representations $\Pi_1,\ldots, \Pi_r$ on $\GL_{4g}(\A_E)$ such that
\begin{align*}
\wedge^2 \Pi' &= \AI_{E/F}(\omega)|\cdot|^{k_0-1} \boxplus \Pi_1 \boxplus \cdots \boxplus \Pi_r;\\
\Sym^2(\Pi') &= \AI_{E/F}(\Sym^2(\pi)) \boxplus \Pi_1 \boxplus \cdots \boxplus \Pi_r.
\end{align*}
The product $L^S(s, \Pi_1\otimes \omega_{|E}^{-1}) \cdots L^S(s, \Pi_r\otimes \omega_{|E}^{-1})$ is an entire function. 
Therefore, if $L^S(s, \Sym^2(\Pi')\otimes \omega_{|E}^{-1})$ has a pole at $s = 1$ then  $L^S(s, \AI_{E/F}(\Sym^2(\pi))\otimes  \omega_{|E}^{-1})$ has a pole at $s = 1$.  
Similarly, if $L^S(s, (\wedge^2 \Pi')\otimes \omega_{|E}^{-1})$ has a pole at $s = 1$, then so does $L^S(s, \AI_{E/F}(\omega)|\cdot|^{k_0-1}\otimes  \omega_{|E}^{-1})$. 
The latter is possible if and only if $\omega$ factors through the norm map $\mathrm{N}_{F/E}: F \to E$, i.e. only if $\omega$ is a base change from $E$.

Conversely, assume that $\omega$ factors through the norm map $\mathrm{N}_{F/E}$, and also that $L^S(s, (\wedge^2 \Pi')\otimes \omega_{|E}^{-1})$ does not have a pole at $s = 1$. 
 By \cite[Theorem 1.1]{sha97}, the $L$-series $L^S(s, (\wedge^2\Pi')\otimes \omega_{|E}^{-1})$ and $L^S(s, \Sym^2(\Pi')\otimes \omega_{|E}^{-1})$ 
are always non-zero at $s = 1$. Hence $L^S(s, \Pi_1\otimes \omega_{|E}^{-1}) \cdots L^S(s, \Pi_r\otimes \omega_{|E}^{-1})$ has a zero of order $g$
at $s = 1$ since $L^S(s, \AI_{E/F}(\omega)|\cdot|^{k_0-1}\otimes  \omega_{|E}^{-1})$ has a pole of order $g$ at $s = 1$.
This means that $L^S(s, \Sym^2(\Pi')\otimes \omega_{|E}^{-1})$ must have a simple pole at $s = 1$, or equivalently, that 
$L^S(s, \AI_{E/F}(\Sym^2(\pi))\otimes  \omega_{|E}^{-1})$ must have a pole of order $g + 1$ at $s = 1$. 

Since $\omega$ factors through the norm map, we have 
$$\AI_{E/F}(\Sym^2(\pi))\otimes  \omega_{|E}^{-1} = \AI_{E/F}(\Sym^2(\pi)\otimes\omega^{-1}).$$
By Theorem~\ref{thm:auto-induction}, if $\Sym^2(\pi)\otimes \omega^{-1}$ is cuspidal, then so is $\AI_{E/F}(\Sym^2(\pi))\otimes  \omega_{|E}^{-1}$. 
Therefore, if $L^S(s, \AI_{E/F}(\Sym^2(\pi))\otimes  \omega_{|E}^{-1})$
has a pole at $s = 1$ then the same is true for $L^S(s, \Sym^2(\pi)\otimes \omega^{-1})$. 
However, by \cite[Theorem 5.1]{tak14}, $L^S(s, \Sym^2(\Pi')\otimes \omega_{|E}^{-1})$ and $L^S(s, \Sym^2(\pi)\otimes \omega^{-1})$
can have at most simple poles at $s = 1$. Therefore, the order of the pole of $L^S(s, \AI_{E/F}(\Sym^2(\pi))\otimes  \omega_{|E}^{-1})$
at $s = 1$ is at most $g$. So, $L^S(s, \Sym^2(\Pi')\otimes \omega_{|E}^{-1})$ cannot have a pole at $s = 1$, which is a contradiction. 
So, we conclude that $\Pi'$ satisfies the conditions of Theorem~\ref{thm:descent} if and only if $\omega$ is a base change from $E$. 
Therefore, $\Pi'$ descends to a globally generic, cuspidal, automorphic representation $\Pi$ of $\GSpin_{2g+1}(\A_E)$ if and  only if 
$\omega$ factors through the norm map $\mathrm{N}_{F/E}: F \to E$. 
\end{proof}

\begin{rem}\rm Let $F/E$ be a cyclic extension of totally real fields of degree $g$. Let $\pi$ be a cuspidal Hilbert automorphic representation, 
with central character $\omega$, which is not a base change from any sub-extension $F'/E$ of $F$. Assume that $\omega$ is {\it not} a base 
change from $E$. Then, Theorem~\ref{thm:theta-lifts} also shows that $\pi$ lifts to an automorphic representation $\Pi$ on the even split
$\GSpin_{2g}$ or one of its quasi-split non-split form $\GSpin_{2g}^*$. The central character of $\Pi'$ is $\omega_{|E}$, so  we obtain a quadratic
character $\mu: E^\times\backslash\A_E^\times \to \{\pm 1\}$ given by $$\mu = \omega_{\Pi'}\omega_{|E}^{-g} = \omega_{|E}^{1-g}.$$
If $\mu$ is trivial, then $\pi$ lifts to the split $\GSpin_{2g}$. Otherwise, it lifts to the quasi-split non-split $\GSpin_{2g}^*$ determined by $\mu$.
\end{rem}

\begin{cor}\label{cor:jlr-lifts} Let $F/E$ be a quadratic extension of totally real fields, and $\pi$ a cuspidal Hilbert automorphic representation
of $\GL_2(\A_F)$, which is not a base change from $E$. Let $\omega$ be the central character of $\pi$, and assume that it is of the form 
$\omega = \mathrm{N}_{F/E}\circ \omega'$ for some character 
$\omega':E^\times \backslash \A_E^\times \to \C^\times$. Then, there exists a holomorphic Hilbert-Siegel automorphic representation
$\Pi^{\rm h}$ of $\GSp_4(\A_E)$ which lifts $\pi$ in the sense that its functorial transfer to $\GL_4(\A_E)$ is $\Pi' = \AI_{E/F}(\pi)$. 
\end{cor}

\begin{proof} We use the same argument as in the proof of \cite[Proposition 15.4]{pil17}.
Let $\Pi$ be the lift of $\pi$ obtained from Theorem~\ref{thm:theta-lifts}, and write 
$$\Pi = \Pi_\infty^{\rm g} \otimes \Pi_f = \left(\bigotimes_{\nu \mid \infty}\Pi_\nu^{\rm g}\right) \otimes \Pi_f,$$
where $\Pi_\infty^{\rm g}$ and $\Pi_f$ are the infinite and finite parts 
of $\Pi$, respectively. This is a globally generic, cuspidal, automorphic representation of $\GSpin_{5}(\A_E) \simeq \GSp_4(\A_E)$. 
Therefore, the global A-packet of $\Pi$ is generic in the sense of \cite[The Classification Theorem]{art04}.
Hence, it is stable and tempered. This means that every local component of this A-packet is an L-packet. 
For every $\nu \in J_E$, the L-packet of $\Pi_\nu^{\rm g}$ is $\{\Pi_\nu^{\rm g}, \Pi_\nu^{\rm h} \}$, where
$\Pi_\nu^{\rm h}$ is the holomorphic (limit of) discrete series of the same weight as $\Pi_\nu^{\rm g}$ (see \cite{sch15}). Therefore, by setting 
$$\Pi^{\rm h} = \Pi_\infty^{\rm h} \otimes \Pi_f = \left(\bigotimes_{\nu \mid \infty}\Pi_\nu^{\rm h}\right) \otimes \Pi_f,$$
we obtain an holomorphic automorphic representation $\Pi^{\rm h}$ of $\GSp_4(\A_E)$, which belongs to the same A-packet as $\Pi$.
\end{proof}

\begin{rem}\rm For $F$ real quadratic, Johnson-Leung-Roberts \cite{jlr12} construct a lift to $\GSp_4(\A_\Q)$ whose infinite component is holomorphic.
This lift, also known as the twisted Yoshida lift has been studied extensively, see \cite[Theorem 8.2]{rob01}, \cite[Th\'eor\`eme 6.2]{vig84} and \cite[Section 7.3]{mt02}. 
It is usually constructed using theta correspondence. Corollary~\ref{cor:jlr-lifts} gives a different way of obtaining this lift.  
\end{rem}

\subsection{The weight of the lift}  
Let $G = \GSpin_{2g+1}$ and $G' = \GL_{2g}$, and recall that $\widehat{G} = \GSp_{2g}$ and $\widehat{G'} = \GL_{2g}$. 
(See \cite{as14} or \cite{mt02} for a nice description of these groups in terms of their root data.) Let$(B, T)$ be a pair of a Borel subgroup $B$ and a 
maximal torus $T$ in $G$. Also let $(B',T')$ be the pair consisting of the standard Borel and maximal torus of $G'$.
 We have the pairs $(\widehat{B}, \widehat{T})$ and $(\widehat{B'}, \widehat{T'})$ for the dual groups $\widehat{G}$ and
 $\widehat{G'}$, respectively. Let $X^*(\widehat{T})$ and $X^*(\widehat{T'})$ (resp. $X_{*}(\widehat{T})$ and $X_{*}(\widehat{T'})$) 
 be the corresponding character (resp. cocharacter) lattices. The natural inclusion of dual tori $\widehat{T} \hookrightarrow \widehat{T'}$ induces 
 a homomorphism $X_{*}(\widehat{T})\otimes \Q \hookrightarrow X_{*}(\widehat{T'})\otimes \Q$ of cocharacter lattices by composition. 
 Let $\delta$ and $\delta'$ be the half-sum of the positive coroots in
$X_{*}(\widehat{T})$ and $X_{*}(\widehat{T'})$, respectively. 

\begin{lem}\label{lem:weight-of-transfer}
Let $\Pi$ be a globally generic cuspidal automorphic representation of $\GSpin_{2g+1}(\A_E)$ of weight $\ul$, and $\Pi'$ its functorial transfer to $\GL_{2g}(\A_E)$. 
Then, the weight $\ul'$ of $\Pi'$ is the image of $\ul$ under the inclusion map $X_{*}(\widehat{T})\otimes \Q \hookrightarrow X_{*}(\widehat{T'})\otimes \Q$.
\end{lem}

\begin{proof} Let $\phi_\nu$ and $\phi_\nu'$ be the respective Langlands parameters of $\Pi_\nu$ and $\Pi_\nu'$ 
(see~ Section~\ref{sec:auto-descent}), so that we have: 
\begin{align*}
\phi_\nu:\, W_\R \to \GSp_{2g}(\C),\,\,\text{and}\,\, \phi_\nu':\,W_\R \to \GL_{2g}(\C).
\end{align*}
These Langlands parameters are determined by their Harish-Chandra parameters 
which are the orbits of $\eta_\nu := \lambda_\nu + \delta$ and $\eta_\nu' := \lambda_\nu' + \delta'$ under the respective Weyl groups.
The fact that $\iota \circ \phi_\nu = \phi_\nu'$ means that the Weyl orbit of $\eta_\nu'$ is the image of the Weyl orbit
of $\eta_\nu$ under the map $X_{*}(\widehat{T})\otimes\Q \hookrightarrow X_{*}(\widehat{T'})\otimes\Q$. We recall that $X_{*}(\widehat{T}) = X^{*}(T)$
and $X_{*}(\widehat{T'}) = X^{*}(T')$. Since $G$ and $G'$ are of type $B_g$ and $A_{2g-1}$, respectively, one sees that
$\delta$ maps to $\delta'$. Hence $\lambda_\nu$ maps to $\lambda_\nu'$ for all $\nu \in J_E$. Therefore, $\ul$ maps to $\ul'$. 
\end{proof}

\begin{rem}\rm In \cite[\S 6]{as06}, there is an explicit recipe which relates the Langlands pa\-ra\-me\-ters of $\Pi$ and $\Pi'$. Lemma~\ref{lem:weight-of-transfer}
can be seen as a reinterpretation of this recipe in terms of weights. Similarly, for $G = \GSp_4$ and $G' = \GL_4$, and $\Pi$ a holomorphic cuspidal
Hilbert-Siegel automorphic representation on $G$, there is a recipe in \cite[\S 2]{sor10}, which gives the Langlands parameters of its 
functorial transfer $\Pi'$ to $G'$ in terms of the Harish-Chandra parameters of $\Pi$ using theta correspondence. This recipe can also be recovered from
Lemma~\ref{lem:weight-of-transfer} and the (accidental) isomorphism $\widehat{G} = \GSpin_5 \simeq G$. 
\end{rem}

\begin{cor}\label{cor:weight-of-transfer} 
Let $\pi$ be a cuspidal Hilbert automorphic representation of weight $\uk$ which satisfies the conditions of Theorem~\ref{thm:lifts}. 
Let $\uk_E = (k_{E, \nu})_{\nu \in J_E}$, where $k_{E,\nu} = (k_{\tau})_{\tau \in S_{\nu}}$ and 
$S_{\nu} = \left\{\tau \in J_F :\, \tau_{|E} = \nu\right\}$. Up to labelling and reordering the set of places in $S_\nu$, we can assume
that $k_{E, \nu} = (k_{\nu,1},\,\ldots,\,k_{\nu,g})$ such that $k_{\nu,1}\ge k_{\nu,2} \ge \cdots \ge k_{\nu, g}\ge 2$. Then, the weight
$\ul$ of $\Pi$ is the inverse image, under the inclusion map $X_{*}(\widehat{T})\otimes\Q \hookrightarrow X_{*}(\widehat{T'})\otimes\Q$, 
of the weight $\ul' = (\lambda_\nu')_{\nu \in J_E}$ of $\Pi'$ given by 
$$\lambda'_\nu = (k_{\nu,1} + m_{\nu, 1} - 2g, \ldots, k_{\nu,g} + m_{\nu, g} - g - 1, m_{\nu, g} - g + 1,\ldots, m_{\nu,1}).$$
\end{cor}

\begin{proof} For each prime $\lambda$ of $L_{\pi}$, the Hodge-Tate weights of $r_{\lambda}(\Pi')$ are given by
$$\mathrm{HT}_{\nu}(r_\lambda(\Pi')) = \{m_{\nu, 1},\ldots, m_{\nu,g}, k_{\nu,g} + m_{\nu, g} - 1, \ldots, k_{\nu,1} + m_{\nu, 1} - 1\}.$$
Therefore, the weight of $\Pi'$ is given by $\ul' = (\lambda_\nu')_{\nu \in J_E}$ where 
$$\lambda'_\nu = (k_{\nu,1} + m_{\nu, 1} - 2g, \ldots, k_{\nu,g} + m_{\nu, g} - g - 1, m_{\nu, g} - g + 1,\ldots, m_{\nu,1}).$$
We conclude by applying Lemma~\ref{lem:weight-of-transfer}. 
 \end{proof}

\begin{rem}\rm Corollary~\ref{cor:weight-of-transfer} implies that the weight of the lift $\Pi$ in Theorem~\ref{thm:theta-lifts}
is regular if and only if the $k_\tau$'s are pairwise distinct, in which case $k_0 \ge 2g$. In particular, automorphic inductions of 
cuspidal Hilbert automorphic representation of parallel weight $2$ are always limits of discrete series. 
\end{rem}



\subsection{\bf Level structure and new vectors theory for the lift} Let $\gN$ be the level of the Hilbert automorphic representation $\pi$ in Theorem~\ref{thm:theta-lifts}.
Then, the level structure of $\Pi$ is the paramodular group $K^{\rm par}(\gN')$ where $\gN' = \mathrm{N}_{F/E}(\gN)\mathfrak{D}_{F/E}^g$ and 
$\mathfrak{D}_{F/E}$ is the different of the relative extension $F/E$. We refer to \cite{gross15} for the description of $K^{\rm par}(\gN')$ as
well as the theory of new vectors. Although the discussion in there is concerned with $\SO_{2g+1}$ over $\Q$, the results extend to $\GSpin_{2g+1}$ 
over general number fields.

\section{\bf Descent of isogeny classes of abelian varieties}\label{sec:descent-ab-av}

Let $A$ be an abelian variety over $F$, with $L = \End_{F}(A)\otimes \Q$.
Let $\lambda$ be a prime in $L$, and denote by $\rho_{A, \lambda}: \Gal(\Qbar/F) \to \GSp_{2g}(L_{\lambda})$  the Galois 
representation into the $\lambda$-adic Tate module of $A$. If $A$ is of $\GL_2$-type, in which case $\dim(A) = [L:\Q]$, then $g = 1$, and we have 
$\rho_{A, \lambda}: \Gal(\Qbar/F) \to \GL_2(L_{\lambda})$. In that case, we let
$$ a_{\gp} := \Tr(\rho_{A, \lambda}(\Frob_{\gp})).$$
Assume that $F/E$ is a Galois extension, and let $\rho_{\lambda}: \Gal(\Qbar/F) \to \GL(V_{\lambda})$ be a $\lambda$-adic representation.
There is a natural action of $\Gal(F/E)$ on the set of such representations given by
$$(\sigma\cdot \rho_{\lambda})(\Frob_{\gp}) := \rho_{\lambda}(\Frob_{\sigma(\gp)}).$$

\begin{thm}\label{thm:av-descent1} Let $F/E$ and $L/K$ be extensions of totally real number fields of
degree $g$. Let $B$ be an abelian variety of dimension $[L:\Q]$ defined over $E$. 
Suppose that the followings hold:
\begin{enumerate}[1.]
\item $F$ is a cyclic extension of $E$ of degree $g$;
\item $\End_{E}(B)\otimes \Q = K$;
\item $\End_F(A)\otimes \Q = L$, where $A = B \times_E F$, and $B$ does not become of $\GL_2$-type
over any proper sub-extension $F'/E$ of $F$.
\end{enumerate}
Then $L/K$ is cyclic, and there exists a generator $\tau \in \Gal(L/K)$ such that
$$ a_{\sigma(\gp)}  = \tau( a_\gp),\,\,\text{for all primes}\,\,\gp.$$
\end{thm}

\begin{proof} Let $V_\ell$ be the underlying $\Q_\ell$-vector space to the 
$\ell$-adic Tate modules of $A$ and $B$; and
$$\rho_{A, \ell} : \Gal(\Qbar/F) \to \GL(V_\ell),\,\,\text{and}\,\,\rho_{B, \ell} : \Gal(\Qbar/E) \to \GL(V_\ell)$$
the corresponding Galois representations. By Faltings \cite[S\"atzen 3 and 4]{fal83}, $V_\ell$ is a semi-simple 
$\Gal(\Qbar/F)$-module, and we have
$$L \otimes \Q_\ell = \End_F(A) \otimes \Q_\ell = \End_{\Q_\ell[\Gal(\Qbar/F)]}(V_\ell).$$
Let $\lambda \mid \ell$ be a prime of $L$, and $L_\lambda$ the completion of $L$ at $\lambda$, and set
$V_\lambda= V_\ell \otimes_{L\otimes \Q_\ell}L_\lambda$. Since $A$ is of $\GL_2$-type, $V_\lambda$
is a rank two $L_\lambda$-module. Let $\rho_{A, \lambda}: \Gal(\Qbar/F) \to \GL(V_\lambda)$ be
the corresponding $\lambda$-adic Tate module. By \cite[Proposition 3.3]{ribet04},
$V_\lambda$ is absolute irreducible with $\End_{\Q_\ell[\Gal(\Qbar/F)]}(V_\lambda) = L_\lambda$ 
and $\End_{L_{\lambda}}(V_{\lambda}) = L_{\lambda}$.

Similarly, let $\lambda' \mid \ell$ be a prime of $K$, and $K_{\lambda'}$ the completion of $K$ at $\lambda'$,
and set $V_{\lambda'}= V_\ell \otimes_{K\otimes \Q_\ell}K_{\lambda'}$. Since $V_\ell$ has rank $2g$
over $K \otimes \Q_\ell$, $V_{\lambda'}$ is a $K_{\lambda'}$-module of rank $2g$. A similar argument
as in \cite[Proposition 3.3]{ribet04}, shows that $V_{\lambda'}$ is absolute irreducible as a 
$K_{\lambda'}[\Gal(\Qbar/E)]$-module with $\End_{\Q_\ell[\Gal(\Qbar/E)]}(V_{\lambda'}) = K_{\lambda'}$ 
and $\End_{K_{\lambda'}}(V_{\lambda'}) = K_{\lambda'}$. As above, the corresponding $\lambda'$-adic Tate modules 
$\rho_{A,\lambda'}:\,\Gal(\Qbar/F) \to \GL(V_{\lambda'})$ and $\rho_{B,\lambda'}:\,\Gal(\Qbar/E) \to \GL(V_{\lambda'})$ 
can also be written as
$$\rho_{A, \lambda'}: \Gal(\Qbar/F) \to \GSp_{2g}(K_{\lambda'}),\,\,\text{and}\,\, \rho_{B, \lambda'}: \Gal(\Qbar/E) \to \GSp_{2g}(K_{\lambda'}).$$ 
Let $\lambda$ be a prime of $L$ lying above $\lambda'$ in $K$. Since $\rho_{A,\lambda'} = \rho_{B,\lambda'}|_{\Gal(\Qbar/F)}$ is reducible
after extension of scalars to $L_\lambda$, and $\sigma\cdot \rho_{A, \lambda'} = \rho_{A, \lambda'}$, we have that
$$\rho_{A, \lambda'}\otimes_{K_{\lambda'}} L_\lambda = \rho_{A, \lambda} \oplus \sigma\cdot\rho_{A, \lambda}\oplus \cdots \oplus\sigma^{g-1}\cdot\rho_{A, \lambda}.$$ 
Since $B$ doesn't become of $\GL_2$-type over any proper subfield $F'/E$ of $F$, the summands are pairwise inequivalent. 
This means that the representation $\Ind_{E}^{F}(\rho_{A,\lambda})$ is irreducible, and we have 
$$\rho_{B,\lambda'}\otimes_{K_{\lambda'}}L_\lambda =  \Ind_{E}^{F}(\rho_{A,\lambda}).$$
So $\Ind_{E}^{F}(\rho_{A,\lambda})$ is defined over $K$. 

Let $\gp$ be a prime of $F$, and recall that $ a_{\gp} = \Tr(\rho_{A, \lambda}(\Frob_{\gp})) \in L$.
By construction, we have 
$$(\sigma\cdot \rho_{A, \lambda})(\Frob_{\gp}) = \rho_{A, \lambda}(\Frob_{\sigma(\gp)}),$$
and
$$\Tr(\rho_{A, \lambda}(\Frob_{\sigma(\gp)}) = a_{\sigma(\gp)} \in L.$$
We need to show that the map $(\tau:\, L \to L, a_\gp \mapsto a_{\sigma(\gp)})$ is a homomorphism. 
To this end, let $S_K :=\left\{\theta: L \hookrightarrow \Qbar\,\big|\,\theta_{|K} = 1\right\}$.
Then, since $B$ becomes of $\GL_2$-type over $F$, we have 
\begin{align*}
\mathrm{charpoly}(\rho_{A, \lambda'}(\Frob_\gp)) &= \prod_{\theta \in S_K}(x^2 - \theta(a_{\gp}) x + \N\!\gp) =\prod_{i=0}^{g-1}(x^2 - a_{\sigma^i(\gp)} x + \N\!\gp).
\end{align*} 
By Tate's conjecture on endomorphism rings of abelian varieties, the extension $L/K$ is generated by the set $\{a_\gp:\, \gp\,\,\text{prime}\}$ 
(see \cite[Proposition 3.5]{ribet04}). So, the map $\tau$ is indeed a homomorphism. Furthermore, since $\End_{E}(B)\otimes\Q = K$, 
and $B$ doesn't become of $\GL_2$-type over any proper subfield $F'/E$ of $F$, there exists a totally split prime $\gp\subset \CO_F$ 
such that $a_\gp$, and hence $a_{\sigma(\gp)}$, is a primitive element. So, the map $(\phi:\,\Gal(F/E)\to \Aut(L),\,\sigma \mapsto \tau)$ 
is injective. By letting $H = \mathrm{im}(\phi)$, we see that $K = L^H$ and
$\Gal(F/E) \simeq \Gal(L/K)$. This concludes the proof of the theorem. 
\end{proof}

We conclude this section with the following converse statement to Theorem~\ref{thm:av-descent1}.

\begin{thm}\label{thm:av-descent} Let $F/E$ and $L/K$ be cyclic extensions of totally real number fields of
degree $g$. Let $A$ be an abelian variety of dimension $[L:\Q]$ defined over $F$. Suppose that the followings hold:
\begin{enumerate}[1.]
\item $\End_F(A)\otimes \Q = L$;
\item There exists a generator $\tau \in \Gal(L/K) $ such that
\begin{align*}
a_{\sigma(\gp)} = \tau(a_{\gp}),\,\,\text{for all primes}\,\, \gp. 
\end{align*}
\end{enumerate}
Then, there exists an abelian variety $B$ defined over $E$, with $\End_E(B)\otimes\Q = K$,
such that $A \sim B \times_E F$.
\end{thm}

\begin{proof} Let $X = \Res_{F/E}(A)$, $\ell$ a rational prime, and $V_\ell(X)$ the $\ell$-adic Tate
module of $X$. By \cite[Satz 4]{fal83} and \cite[Section 1]{mil72}, we have
\begin{align*}
\End_E(X) \otimes \Q_\ell &= \End_{\Q_\ell[\Gal(\Qbar/E)]}(V_\ell(X))= 
\End_{\Q_\ell[\Gal(\Qbar/E)]}(\Ind_{E}^F(V_\ell(A)))\\
& = \End_{\Q_\ell[\Gal(\Qbar/E)]}(\Ind_{E}^F(\rho_{A, \ell})) = 
\prod_{\lambda' \mid \ell} \End_{K_{\lambda'}[\Gal(\Qbar/E)]}(\Ind_{E}^F(\rho_{A, \lambda'}))\\
& = \prod_{\lambda' \mid \ell} \End_{K_{\lambda'}}(\Ind_{E}^F(\rho_{A, \lambda'})).
\end{align*}
Let $\lambda', \lambda$ be primes in $K$ and $L$, respectively, with $\lambda \mid \lambda'$. Then, we have
$$\rho_{A, \lambda'}\otimes_{K_{\lambda'}} L_\lambda = \rho_{A, \lambda} \oplus \sigma\cdot \rho_{A,\lambda} 
\oplus \cdots \oplus \sigma^{g-1}\cdot\rho_{A, \lambda}.$$
It follows that, we have  
\begin{align*}
\End_{L_{\lambda}}(\Ind_{E}^F(\rho_{A, \lambda'} \otimes_{K_{\lambda'}}L_\lambda)) 
&= \End_{L_\lambda}\left(\bigoplus_{i=0}^{g-1} \Ind_{E}^{F}(\sigma^i\cdot \rho_{A, \lambda}) \right).
\end{align*}
By construction, we have that
$$W_\lambda := \Ind_{E}^{F}(\rho_{A, \lambda}) = \Ind_{E}^{F}(\sigma\cdot \rho_{A, \lambda})= \cdots = \Ind_{E}^{F}(\sigma^{g-1}\cdot \rho_{A, \lambda}).$$
By Condition (2), $W_\lambda$ is an irreducible $L_\lambda[\Gal(\Qbar/E)]$-module, and the characteristic polynomial of every Frobenius element acting on
it lies in $K[x] \subset K_{\lambda'}[x]$. Therefore, there exists an irreducible 
$K_{\lambda'}[\Gal(\Qbar/E)]$-module $W_{\lambda'}$ such that $W_\lambda = W_{\lambda'}\otimes_{K_{\lambda'}} L_\lambda$, and $W_{\lambda'}$
satisfies one of the following two possibilities:
\begin{enumerate}[(i)]
\item  $\End_{K_{\lambda'}[\Gal(\Qbar/E)]}(W_{\lambda'}) = K_{\lambda'}$; or
\item $\End_{K_{\lambda'}[\Gal(\Qbar/E)]}(W_{\lambda'}) = D'$ for the unique division algebra $D'$ of dimension $g^2$ over $K_{\lambda'}$. 
\end{enumerate}
In case (ii), this would mean that there is a simple factor $B$ of $X$, and a totally definite division algebra $D$ of dimension $g^2$
over $K$ such that $\End_E(B)\otimes\Q \simeq D$, and $\End_{K_{\lambda'}[\Gal(\Qbar/E)]}(W_{\lambda'}) = D_{\lambda'}$, where 
$D' = D_{\lambda'}$ is the completion of $D$ at $\lambda'$. Since
$\End_F(X)\otimes\Q = \mathrm{M}_g(L)$, such a division algebra would be split by $L$ which is totally real. But this is impossible,
so case (ii) cannot happen. Hence, $W_\lambda$ descends to $K_{\lambda'}$, and we have 
\begin{align*}
\End_{L_{\lambda}}(\Ind_{E}^F(\rho_{A, \lambda'} \otimes_{K_{\lambda'}}L_\lambda)) 
&=\End_{L_\lambda}\left( (W_{\lambda'}\otimes_{K_{\lambda'}} L_\lambda)^g\right) = \mathrm{M}_g(L_\lambda).
\end{align*}
From this, we see that 
$$\End_{K_{\lambda'}}(\Ind_{E}^F(\rho_{A, \lambda'})) = \End_{K_{\lambda'}}\left( W_{\lambda'}^g\right) = \mathrm{M}_g(K_{\lambda'}).$$
It follows that 
$$\End_E(X) \otimes \Q_\ell = \prod_{\lambda' \mid \ell }\mathrm{M}_g(K_{\lambda'}).$$
Since this is true for every prime $\ell$, we see that $\End_E(X) \otimes \Q = \mathrm{M}_g(K)$. 

Let $B$ be a simple factor of $X$. Then, one sees that $B \times_E F$ is necessarily simple, and that $B \times_E F \sim A$. 
By construction, we have $\End_E(B)\otimes \Q = K$. 
\end{proof}

\begin{rem}\rm We expect Theorems~\ref{thm:av-descent1} and~\ref{thm:av-descent} to generalise to arbitrary Galois extensions
$F/E$. They show that, when $B/E$ is an abelian variety such that $A = B\times_E F$ acquires extra endomorphisms, then there is a strong interaction
between the group structure of $\Gal(F/E)$ and that of $\Aut_K(L)$ where $K = \End_E(B)\otimes \Q$, and $L = \End_F(A)\otimes \Q$. Theorem~\ref{thm:lifts} below
can be seen as a natural counterpart on the automorphic side which indicates the important of this relation from a functorial point of view. 
The $E$-varieties, which have been extensively studied in \cite{ribet04, pyle04, guitart10}, seem to represent an extreme case where 
this relation is simply absent as the action of the group $\Gal(F/E)$ on $L$ is trivial. (We refer to \cite{demb16} for more details on this.) 
\end{rem}

\section{\bf The Gross conjecture}\label{sec:gross-conj}

The goal of this section is to provide some evidence for the following statement which is a slight generalisation of 
Conjecture~\ref{conj:gross0} in the introduction. (Our statement should also be compared with \cite[Conjecture 1.4]{bk14}). 

\begin{conj}[Gross-Langlands]\label{conj:gross} Let $E$ be a number field, and $B$ be an abelian variety defined over $E$ 
such that $\End_E(B)\otimes \Q = K$ is a totally real field of degree $d$ with $\dim(B) = dg$, for some $g \in \Z_{\ge 1}$. 
Assume that $B$ is not of $\GL_2$-type, i.e., that $g \ge 2$ when $\dim(B) \ge 2$. 
Let $\mathrm{cond}(B)=\gN^d$ be the conductor of $B$. Then there exists a globally generic
cuspidal automorphic representation $\Pi$ on $\GSpin_{2g+1}(\A_E)$ of weight $2$ and paramodular level structure $\gN$, 
with field of rationality $K$, such that 
$$L(B, s) = \prod_{\Pi' \in [\Pi]}L(\Pi', s),$$
where $[\Pi]$ denotes the Hecke orbit of $\Pi$.  
\end{conj}

Conjecture~\ref{conj:gross} was probably known to experts in some less imprecise form. To the best of our knowledge, however, Gross \cite{gross15} 
was the first to give a refinement which also predicts the level of $\Pi$ in terms of the conductor of $B$. In \cite{gross15}, after fixing an 
isomorphism $\iota:\,\Qbar_\ell \simeq \C$, the Weil-Deligne representations arising from the $\ell$-adic Tate module attached to 
$B$ are normalised so that they  take values in $\Sp_{2g}(\C)$. So, the conjectured automorphic representation $\Pi$ lives on 
$\SO_{2g+1}(\A_\Q)$. Using the local Langlands  correspondence for $\SO_{2g+1}$ (see \cite{js04}), Gross 
provides a precise recipe for the local component  $\Pi_v$ for the representation $$\Pi = {\bigotimes_v}' \Pi_v.$$ 
The difficult is in proving that $\Pi$ is automorphic and cuspidal.

\begin{rem}\rm Although the infinite components of the automorphic representation $\Pi$ are limits of discrete series,
the theory of newforms developed in \cite[\S\S6-8]{gross15} allows for a more classical definition of the field of rationality
$L_\Pi$ of $\Pi$, and it should be equal to $K$. So, we can define the Hecke orbit of $\Pi$
under that assumption. But we expect that one can also give a definition of $L_\Pi$ that is more intrinsic.
\end{rem}

\begin{rem}\rm The restriction to abelian varieties not of $\GL_2$-type in our formulation of Conjecture~\ref{conj:gross} is not 
essentially. We refer to the discussion following \cite[Proposition 6]{gross15} to see how to the conjecture can be
restated in a way that encompasses abelian varieties of $\GL_2$-type.
\end{rem}

\subsection{Compatibility with Galois descent}  In this subsection, we show that Conjecture~\ref{conj:gross}
is compatible with lifts of Hilbert automorphic representations, which correspond to descent of isogeny
classes on the geometric side. 

\begin{thm}\label{thm:lifts} Let $F/E$ be a cyclic extension of totally real fields of degree $g$, and
$\Gal(F/E) = \langle \sigma \rangle$. Let $\gN$ be an integral ideal such that $\gN^\sigma = \gN$.
Let $f$ be a Hilbert newform of parallel weight $2$, level $\gN$ and trivial central character, with field of coefficients $L$, 
which is not a base change from any proper sub-extension $F'/E$ of $F$. Assume that the Hecke orbit of $f$ is preserved by $\Gal(F/E)$,
and that there is an abelian variety $A$ attached to $f$ by the Eichler-Shimura construction. Then, the followings hold:
\begin{enumerate}[(i)]
\item There exist a subfield $K$ of $L$ such that the extension $L/K$ is cyclic Galois of degree $g$; and a generator 
$\tau \in \Gal(L/K) $ such that, for all primes $\gp$, 
\begin{align*}
a_{\sigma(\gp)}(f) = \tau(a_{\gp}(f)). 
\end{align*}
\item The isogeny class of $A$ descends to $E$, i.e. there exists an abelian variety $B$ defined over $E$
such that $\End_E(B)\otimes \Q \simeq K$ and $A \sim B\times_E F $; 
\item The automorphic representation $\pi$ on $\GL_2(\A_F)$ attached to $f$ lifts to a globally generic 
cuspidal automorphic representation $\Pi$ on $\GSpin_{2g+1}(\A_E)$, with field of rationality $K$, such that 
$$L(B,s) = \prod_{\Pi' \in[\Pi]} L(\Pi', s).$$
\end{enumerate}
\end{thm}

\begin{proof} (i) Let $\pi$ be the automorphic representation attached to $f$, and let ${}^\sigma\!\pi = \pi \circ \sigma$.
Then  for $\gp$ prime, $a_{\sigma(\gp)}(f)$ is the Hecke eigenvalue at $\gp$ of the newform ${}^\sigma\!f \in {}^\sigma\!\pi$.
The newforms $f$ and ${}^\sigma\!f$ are in the same Hecke orbit if and only if ${}^\sigma\!\pi$ and $\pi$ are in 
the same Hecke orbit. This is equivalent to saying that there exists $\tau \in \Gal(L/K)$ non-trivial such that ${}^\sigma\!\pi = \pi^\tau$, 
which translates into the stated identity. Furthermore, since $f$ is not a base change from any intermediate field, $\tau$ must be
a generator. 

(ii) The identity (i) is the descent condition on the isogeny class of $A$ given in Theorem~\ref{thm:av-descent}.
So, the isogeny class of $A$ descends to that of an abelian variety $B$ over $E$.

(iii) Since $\pi$ is not a base change from any proper sub-extension $F'/E$ of $F$, we deduce from Theorem~\ref{thm:theta-lifts} 
that $\pi$ admits a lift $\Pi$ to $\GSpin_{2g + 1}(\A_E)$. By Lemma~\ref{lem:fld-of-rationality}, the field of rationality of $\Pi$ is $K$. 
By construction, we have $$L(A, s) = \prod_{\pi' \in [\pi]}L(\pi', s).$$
So, it follows from the functorial properties of Theorem~\ref{thm:auto-induction}, and the local-global compatibility result stated 
in \cite[Proposition 5.1]{as14}, that 
$$L(B, s) = \prod_{\Pi' \in [\Pi]} L(\Pi', s).$$
\end{proof}

\begin{rem}\label{rem:lifts}\rm In Theorem~\ref{thm:lifts}, the condition that $\pi$ is not a base change from
any proper sub-extension $F'/E$ of $F$ can be relaxed as follows. Let $H = \mathrm{Stab}_{G}([\pi])$,
where $G = \Gal(F/E)$, and $F' = F^{H}$, $g' = [F:F']$. Then, as before, we have an injection
$\phi: H \hookrightarrow \Aut(L)$. Letting $K' = L^{\phi(H)}$, Theorem~\ref{thm:lifts} implies that $\pi$
lifts to an automorphic representation $\Pi$ on $\GSpin_{2g'+1}(\A_{F'})$, and that $A$ descends to
an abelian variety $B$ over $F'$ such that $\End_{F'}(B)\otimes \Q = K'$ and 
$$L(B, s) = \prod_{\Pi' \in [\Pi]}L(\Pi', s).$$ 
\end{rem}

\subsection{Automorphy of abelian varieties acquiring extra endomorphisms}
We now prove a converse statement to Theorem~\ref{thm:lifts}. We refer to \cite{bdj10} for the Serre conjecture for totally real number fields.

\begin{thm}\label{thm:gross-evidence1} Let $F/E$ and $L/K$ be extensions of totally real number fields of degree $g$. 
Let $B$ be an abelian variety of dimension $[L:\Q]$ defined over $E$. Suppose that the followings hold:
\begin{enumerate}[1.]
\item $F$ is a cyclic extension of $E$;
\item $\End_{E}(B)\otimes \Q = K$;
\item $\End_F(A)\otimes \Q = L$, where $A = B \times_{E} F$, but $B$ doesn't become of $\GL_2$-type over any proper subfield $F'/E$ of $F$;
\item There is a rational prime $\ell$, and a prime $\lambda' \mid \ell$ in $K$ which is totally ramified in $L$ such that $\bar{\rho}_{A,\lambda}$ is surjective for
the unique prime $\lambda\mid \lambda'$ in $L$. If $\ell = 2$, then further assume that $F(A[\lambda])$ is non-solvable.
\end{enumerate}
If the Serre conjecture for totally real number fields is true for $E$, then $B$ is automorphic. More specifically, there exists a cuspidal Hilbert automorphic representation $\pi$
on $\GL_2(\A_F)$, which lifts to a globally generic cuspidal automorphic representation $\Pi$ on $\GSpin_{2g+1}(\A_E)$ such that
$$L(B, s) = \prod_{\Pi' \in [\Pi]} L(\Pi', s).$$
\end{thm}

\begin{proof} Keeping the same notations as in Theorem~\ref{thm:av-descent1}, we have the commutative diagram  
\begin{eqnarray*}
\begin{tikzcd}
\Gal(\Qbar/F) \arrow[rightarrow]{r}{\rho_{A, \lambda'}}\arrow[rightarrow, "\mathbf{1}\ltimes\rho_{A, \lambda}"']{dr} & \GSp_{2g}(K_{\lambda'})\\
&\GL_2(L_{\lambda})^g\arrow[rightarrow]{u}\\
\end{tikzcd}
\end{eqnarray*}
where $\mathbf{1}\ltimes\rho_{A, \lambda}:\, \Gal(\Qbar/F)\to \GL_2(L_{\lambda})^g$ is given by
$$(\mathbf{1}\ltimes \rho_{A, \lambda})(\Frob_{\gp}) := (\rho_{A, \lambda}(\Frob_{\gp}), \ldots, \rho_{A, \lambda}(\Frob_{\sigma^{g-1}(\gp)}) ).$$
We recall that, by Theorem~\ref{thm:av-descent1}, there exists a generator $\tau \in \Gal(L/K)$ such that
$$ a_{\sigma(\gp)}  = \tau(a_{\gp}),$$
for all primes $\gp \subset \CO_F$.
Let $\bar{\rho}_{A, \lambda}$ be the reduction of $\rho_{A, \lambda}$ modulo $\lambda$,
and $\bar{\rho}_{A, \lambda}^{ss}$ its semi-simplification. Since $\lambda$ is totally ramified in $L$, we see that
$$a_{\sigma(\gp)}  = a_{\gp}\,\bmod \lambda,$$
for all primes $\gp \subset \CO_F$. So, $\bar{\rho}_{A, \lambda}^{ss}:\, \Gal(\Qbar/F) \to \GL_2(\F_\lambda)$ is a base change from $E$. 
By assumption $\bar{\rho}_{A, \lambda}$ is surjective. So it is absolutely irreducible, hence $\bar{\rho}_{A, \lambda} = \bar{\rho}_{A, \lambda}^{ss}$.
Since we assume that the Serre conjecture for totally real number fields is true for $E$ (see \cite{bdj10}), $\bar{\rho}_{A, \lambda}$ is modular. 
If $\ell \ge 3$, then $\rho_{A, \lambda}$ is modular by \cite[Theorem 1.1]{kt16}. If $\ell = 2$, then $\rho_{A, \lambda}$ is also modular by \cite[Theorem 6.1]{tho17}. 
So, in either case, there exists a Hilbert automorphic representation $\pi$ associated to $A$. 
The abelian varieties $A$ and $B$, and the automorphic representation $\pi$ satisfy the conditions of Theorem~\ref{thm:lifts}. So,
there exists $\Pi$ on $\GSpin_{2g + 1}(\A_E)$ such that 
$$L(B, s) = \prod_{\Pi' \in [\Pi]}L(\Pi', s).$$
\end{proof}

\begin{rem}\rm First, we note that the assumption that the Serre conjecture for totally real fields holds for $E$ in Theorem~\ref{thm:gross-evidence1}
is more than what we need. It suffices that $\bar{\rho}_{A,\lambda}$ be modular.

Second, the conditions in Theorem~\ref{thm:gross-evidence1} are clearly not optimal. Indeed, we assume that the residual 
representation $\bar{\rho}_{A, \lambda}$ is surjective. But, \cite[Theorem 1.1]{kt16} only requires the restriction $\bar{\rho}_{A, \lambda}|_{G_{F(\zeta_\ell)}}$ 
to be irreducible. One could also allow $\bar{\rho}_{A, \lambda}$ to be reducible and use \cite[Theorem 7.1]{tho15} or  \cite[Theorem A]{sw99}. 
In other words, Theorem~\ref{thm:gross-evidence1} can be easily extended to a wider class of abelian varieties by combining some of the most recent 
modularity lifting results, see for example \cite{gee06, gee09, kis09, kt16, sw99, tho15, tho17}. 
\end{rem}

\begin{rem}\rm  In light of recent works of Pilloni \cite{pil17} and Calegari-Geraghty \cite{cg16}, one might
also be able to extend our approach to other groups such as $\GSp_4$ to produce further evidence 
to Conjecture~\ref{conj:gross} via functoriality. 
\end{rem}

\begin{cor}\label{cor:gross-evidence1} Let $F/\Q$ and $L/K$ be extensions of totally real number fields of degree $g$. 
Let $B$ be an abelian variety of dimension $[L:\Q]$ defined over $\Q$. Suppose that the followings hold:
\begin{enumerate}[1.]
\item $F$ is cyclic;
\item $\End_\Q(B)\otimes \Q = K$;
\item $\End_F(A)\otimes \Q = L$, where $A = B \times_\Q F$, but $B$ doesn't become of $\GL_2$-type over any proper subfield $F'/E$ of $F$;
\item There is a rational prime $\ell$, and a prime $\lambda' \mid \ell$ in $K$ which is totally ramified in $L$ such that $\bar{\rho}_{A,\lambda}$ is surjective for
the unique prime $\lambda\mid \lambda'$ in $L$. If $\ell = 2$, then further assume that $F(A[\lambda])$ is non-solvable.
\end{enumerate}
Then $B$ is automorphic, i.e., there exists a cuspidal Hilbert automorphic representation $\pi$
on $\GL_2(\A_F)$, which lifts to a globally generic cuspidal automorphic representation $\Pi$ on $\GSpin_{2g + 1}(\A_\Q)$, such that
$$L(B, s) = \prod_{\Pi' \in [\Pi]} L(\Pi', s).$$
\end{cor}

In particular, we have the following corollary for abelian varieties of prime dimension. 

\begin{cor}\label{cor:prime-degree} Let $B$ be an abelian variety defined over $\Q$ such that $\End_\Q(B)\otimes \Q = K$
is a totally real field. Suppose that there exist totally real extensions $F/\Q$ and $L/K$ of prime degree $g$ such that 
\begin{enumerate}[1.]
\item $F$ is cyclic;
\item $[L:\Q] = \dim(B)$ and $\End_F(A)\otimes \Q = L$, where $A = B \times_{\Q} F$;
\item There is a rational prime $\ell$, and a prime $\lambda'\mid \ell$ in $K$ which is ramified in $L$ such that $\bar{\rho}_{A,\lambda}$ is surjective for
the unique prime $\lambda\mid \lambda'$ in $L$. If $\ell = 2$, then further assume that $F(A[\lambda])$ is non-solvable.
\end{enumerate}
Then, $B$ is automorphic, i.e. there exists a cuspidal Hilbert automorphic representation $\pi$ on $\GL_2(\A_F)$,
which lifts to a globally generic cuspidal automorphic representation $\Pi$ on $\GSpin_{2g+1}(\A_\Q)$,  such that 
$$L(B, s) = \prod_{\Pi' \in [\Pi]}L(\Pi', s).$$
\end{cor}

\begin{proof}[Proof of Corollary~\ref{cor:gross-evidence1}] Since $E = \Q$, the Serre conjecture is true \cite{ser87, kw09}. 
So we conclude by applying Theorem~\ref{thm:gross-evidence1}.
\end{proof}

\begin{rem}\rm There are two separate aspects to Corollary~\ref{cor:gross-evidence1} (and in fact Theorem~\ref{thm:gross-evidence1}).
The first one is concerned with potential modularity of abelian varieties which acquire extra endomorphisms after base change, 
while the second is about functoriality. If one is only interested in potential modularity, then the assumption that the extension $F/\Q$ be
cyclic is not strictly necessary. Indeed, there are results of Dieulefait on modularity of residual Galois representations that are non-solvable base change
(see for example \cite[\S 6]{die12}, and also \cite{hid09}).
\end{rem}

One easily deduces the following uniformisation result from Corollary~\ref{cor:gross-evidence1}. 

\begin{cor}\label{cor:prime-degree2} Keeping the hypotheses of Corollary~\ref{cor:gross-evidence1}, assume that
$F$ has narrow class number one. Let $\mathrm{cond}(A) = \gN^g$ be the conductor of the abelian variety $A$. If $g$ is even, assume that 
there exists a prime $v \mid \gN$ such that $\pi_v$, the local component of the automorphic representation $\pi$ at $v$, is a discrete series 
representation. Let $D$ be the unique quaternion algebra $D$ over $F$ ramified exactly at 
\begin{itemize}
\item $g-1$ of the real places if $g$ is odd; and 
\item $v$ and $g-1$ of the real places of $F$ if $g$ is even.
\end{itemize}
Let $\CO$ be an Eichler order of level $\gN$, and $X_0^D(\gN)$ the
Shimura curve attached to $\CO$. Then, there is a morphism $\phi:\,\Jac(X_0^D(\gN)) \to B$ defined over $\Q$. 
\end{cor}

\begin{proof} The conditions of Corollary~\ref{cor:prime-degree2} ensure that there is a non-constant morphism $\phi:\,\Jac(X_0^D(\gN)) \to A$. 
We want to show that $\phi$ descends to $\Q$. 
First, we note that since $\gN = \gN^\sigma$, and $F$ has narrow class number one, the field of moduli of 
$X_0^D(\gN)$ is $\Q$ (see \cite{doi-naga67}). Let $\Jac(X_0^D(\gN))^{\rm new}$ be the new part of $\Jac(X_0^D(\gN))$, 
and consider the decomposition into isogeny classes over $F$:
$$ \Jac(X_0^D(\gN))^{\rm new} \sim \prod_{[f]}A_f \sim \prod_{[f] \in \mathcal{E}} \left(\prod_{[h] \in G\cdot[f]}A_{h}\right) \sim \prod_{[f] \in \mathcal{E}} B_{f},$$
where $S_2(\gN)^{\rm new}$ is the new subspace of cusp forms of level $\gN$ and weight $2$ over $F$, and $\mathcal{E}$ a set of representatives of the 
Hecke orbits of the eigenforms in $S_2(\gN)^{\rm new}$ under the action of $G = \Gal(F/\Q)$. By Theorem~\ref{thm:lifts} and Remark~\ref{rem:lifts}, 
the abelian variety $B_f$ attached to the Galois orbit $G\cdot[f]$ descends to $\Q$. So, the decomposition of $\Jac(X_0^D(\gN))^{\rm new}$ into 
Galois orbits of its simple factors descends to $\Q$. The Hecke orbit of the newform attached to $A$ is unique in its Galois orbit. So, in particular, 
the morphism $\phi:\,\Jac(X_0^D(\gN)) \to A$ descends to $\Q$. 
\end{proof}

\begin{rem}\rm Although the morphism $\phi:\,\Jac(X_0^D(\gN)) \to B$ is defined over $\Q$,
the abelian variety $B$ itself cannot be the quotient of the Jacobian of a Shimura curve
arising from a quaternion algebra $D$ defined over $\Q$ as $B$ is not of $\GL_2$-type.
\end{rem}

\section{\bf Examples}\label{sec:examples}
In this section, we illustrate Theorems~\ref{thm:lifts} and~\ref{thm:gross-evidence1} with some concrete examples. 

\subsection{An example over the real cubic field of smallest discriminant} \label{ex:zeta7}
Let $F = \Q(\zeta_7)^+ = \Q(b)$, where $b^3 + b^2 - 2b - 1=0$, be the totally real cubic subfield 
of the cyclotomic field generated by the $7$-th root of unity. This is the totally real cubic field of smallest discriminant.
It is Galois, we let $\sigma$ be a generator of $\Gal(F/\Q)$. Let $D$ be the quaternion algebra over $F$ ramified
at exactly 2 of the real places. Let $\CO_D$ be a maximal order in $D$, and $\CO \subset \CO_D$ an Eichler order of level $\gq^2$
where $\gq$ is the inert prime above $3$ in $F$. Let $X_0^D(\gq^2)$ be the Shimura curve attached to $\CO$; it has genus $10$. 
Let $\mathrm{Jac}(X_0^D(\gq^2))$ be its Jacobian. The new part $\mathrm{Jac}(X_0^D(\gq^2))^{\rm new}$ is an $8$-dimensional abelian variety. 

\begin{prop}\label{prop:zeta7} Let $C/\Q$ the hyperelliptic curve given by 
$$C:\, y^2 = -7x^7 + 98x^6 - 1470x^5 + 4606x^4 - 1715x^3 - 4116x^2 - 4459x + 4802,$$
with discriminant $\mathrm{disc}(C) = -2^{36}\cdot 3^{24}\cdot 7^{28}$.
Let $B = \mathrm{Jac}(C)$ be its Jacobian, and $A = B \times_{\Q} F$. Then, we have the followings:
\begin{enumerate}[(i)]
\item $\End_{\Q}(B) = \Z$, and $\End_F(A)\otimes \Q \simeq F$;
\item There exists an automorphic representation $\Pi$ on $\GSpin_7(\A_{\Q})$ such that 
$$L(B, s) = L(\Pi, s);$$ 
\item There is a non-constant morphism $\phi: \Jac(X_0^D(\gq^2)) \to B$ defined over $\Q$. 
\end{enumerate}
\end{prop}

\begin{table}\small
\caption{Hecke eigenvalues for the newforms of level $\gq^2$ over the cubic subfield of
$\Q(\zeta_{7})^{+}$ with coefficient field $\Q(\zeta_{7})^{+}$.} 
\label{table:zeta7}
\begin{tabular}{ >{$}c<{$}   >{$}r<{$}  >{$}r<{$}   >{$}r<{$} }
\toprule
\mathrm{N}\gp&\gp&a_{\gp}(g)&a_{\gp}(g')\\
\midrule
7 & b^2 + 2b - 1 & 2 & 2\\
8 & 2 & 3 & -3\\
13 & -2b^2 - b + 2 & 3b + 2 & 3b + 2\\
13 & b^2 + 2b - 2 & 3b^2 - 4 & 3b^2 - 4\\
13 & b^2 - b - 3 & -3b^2 - 3b + 5 & -3b^2 - 3b + 5\\
27 & 3 & 0 & 0\\
29 & b^2 - 2b - 3 & 3b^2 - 3b - 9 & -3b^2 + 3b + 9\\
29 & -b^2 - 3b + 1 & -6b^2 - 3b + 6 & 6b^2 + 3b - 6\\
29 & -2b^2 + b + 4 & 3b^2 + 6b - 6 & -3b^2 - 6b + 6\\
41 & b^2 - 2b - 4 & -3b^2 - 6b & 3b^2 + 6b\\
41 & 2b^2 + 3b - 4 & 6b^2 + 3b - 12 & -6b^2 - 3b + 12\\
41 & -3b^2 - b + 3 & -3b^2 + 3b + 3 & 3b^2 - 3b - 3\\
 \bottomrule
\end{tabular}
\end{table}

\begin{proof} We use the algorithm described in \cite{cmsv17} to show that $A$ has RM by a suborder of $\CO_F$
of index $3$. A \verb|Magma| calculation shows that the prime-to-$2$-part of the conductor of $A$ is $\gq^6$. 

Let $\lambda$ be the unique prime above $7$. By computing the orders of Frobenii for the first few primes in $F$, one sees that 
$\bar{\rho}_{A, \lambda}:\, \Gal(\Qbar/F) \to \GL_2(\F_7)$ is surjective. If $B$ were of $\GL_2$-type, then $A$ would be modular 
as a base change. Otherwise, we can apply Corollary~\ref{cor:prime-degree}. So, in either cases, $A$ is modular. 

Now, we will show that $B$ cannot be of $\GL_2$-type. To this end, let $C'$ be the quadratic twist of $C$ by $d = 8$, and let 
$B' = \mathrm{Jac}(C')$ and $A' = B'\times_{\Q} F$. Then $A'$ is also modular. It is much easier to compute the conductor of
$A'$, and we obtain that $\mathrm{cond}(A') = (2^{18}\cdot 3^6) = {\gN'}^3$. Since $A'$ is the quadratic twist of $A$ by $d = 8$,
we get that $\mathrm{cond}(A) = \gq^6 = \gN^3$. So $A$ has good reduction at $2$. 

If $B$ were of $\GL_2$-type, then the Hilbert newform corresponding to $A$
would be a base change. The new subspace of Hilbert cusp forms of level $\gN$ and weight $2$ is $8$-dimensional 
(as explained above). It has two Hecke constituents of dimension $3$. In Table~\ref{table:zeta7}, we have listed their
Hecke eigenvalues for the first few primes. One can see that they are inner twist of each other, and neither of them is a 
base change. Therefore $B$ cannot be of $\GL_2$-type. Since $F$ is cubic, this means that $\End_{\Q}(B) = \Z$. 
The same shows that $\End_\Q(B') = \Z$. We conclude the proof using Corollaries~\ref{cor:prime-degree} and~\ref{cor:prime-degree2}.  
\end{proof}

\begin{rem}\rm The argument we used to prove that $B$ is not of $\GL_2$-type cannot be applied directly to $B'$ as the
dimension of the new subspace of cusp forms of level $\gN'$ and weight $2$ is $1884288$. 
\end{rem}

\subsection{An example over the real cubic subfield of $\Q(\zeta_{43})^+$} \label{ex:real-cubic-43} 
Our second example is over the cubic subfield $F = \Q(b)$ of $\Q(\zeta_{43})^{+}$, here $b^3 + b^2 - 14b + 8=0$. 
We consider the quaternion algebra $D$ ramified at exactly two of the real places of $F$. We let $\CO_D$ be a maximal order, and $X_0^D(1)$ the Shimura
curve of level $(1)$. This is a curve of genus $6$ defined over $\Q$. Under the Hecke action, we have the decomposition
$$\mathrm{Jac}(X_0^D(1)) \sim E \times A \times A',$$
where $E$ is the base change of the elliptic curve defined over $\Q$ by
\begin{align*}
y^2 + xy + y = x^3 - 39x - 27,
\end{align*}
with $j$-invariant $j(E)=1849=43^2$. The abelian surface $A$ has RM by $\Z[\frac{1+\sqrt{5}}{2}]$. It is the Jacobian of the genus $2$ 
curve $C:\, y^2 +  Q(x)y = P(x)$ given by 
\begin{align*}
P(x) &:=(9b^2 - 33b + 20)x^6 + (24b^2 - 88b + 54)x^5 + (42b^2 - 154b + 94)x^4\\
&\qquad + (42b^2 - 154b + 94)x^3+ (21b^2 - 77b + 47)x^2 + (6b^2 - 22b + 13)x\\
&\qquad - 15b^2 + 55b - 34;\\
Q(x) &:= x^3+x+1.
\end{align*} We have $\mathrm{disc}(C) = u(-b^2 + 5b - 3)^{22}$, where $-b^2 + 5b - 3$ generates the unique prime ideal above $43$
and $u\in \CO_F^\times$, and $\mathrm{cond}(A) = (1)$. In fact, $A$ is a quadratic twist of the base change of the Jacobian of
the curve $C'$ given by
$$C':\, y^2 - (x^3 + x  + 1)y = -x^6 - 2x^5 - 4x^4 - 4x^3 - 2x^2 - x + 1.$$

\begin{prop}
There exists an abelian threefold $B'$ defined over $\Q$ such that $\End_{\Q}(B') = \Z$, and $A' = B'\times_{\Q} F$ becomes
of $\GL_2$-type, with {\rm RM} by the cubic subfield $K = \Q(e)$ of $\Q(\zeta_{19})^{+}$ given by $e^3 - e^2 - 6e + 7=0$. So, the 
decomposition $$\mathrm{Jac}(X_0^D(1)) \sim E \times A \times A'$$
descends to $\Q$.
\end{prop}

\begin{proof} The Hecke eigenvalues of the newform corresponding to $A$ are listed in Table~\ref{table:eigenform-zeta43}.
One see that it cannot be a base change.  Since it belongs to the unique Hecke constituent of dimension $3$, it must
satisfy the conditions of Theorem~\ref{thm:lifts}. It follows that its isogeny class descend to an abelian threefold $B'$
with $\End_{\Q}(B') = \Z$. Since $E$ and $A$ also descend to $\Q$, we conclude from Corollary~\ref{cor:prime-degree2}
that the decomposition of $\mathrm{Jac}(X_0^D(1))$ must descend as well. 
\end{proof}

\begin{rem}\rm From the Hecke eigenvalues listed in Table~\ref{table:eigenform-zeta43}, it is not hard to check that
the mod $2$ Galois representation attached to $A'$ surjects onto $\PSL_2(\F_{8})$. Since this is not a subgroup 
of $S_{8}$, we see that $B'$ cannot be the Jacobian of a hyperelliptic curve (assuming that it is principally polarised).
\end{rem}

\begin{table}
\caption{Euler factors for the threefold $A'$ of level $(1)$ over the cubic subfield of
$\Q(\zeta_{43})^{+}$ with RM by the cubic subfield of $\Q(\zeta_{19})^{+}$.} 
\label{table:eigenform-zeta43}
\scalebox{0.88}{
\begin{tabular}{ >{$}c<{$}   >{$}r<{$}  >{$}r<{$}   >{$}r<{$} }
\toprule
\mathrm{N}\gp&\gp&a_{\gp}(g)&\text{Euler factor}\\
\midrule
2 & \frac{1}{2}(b^2 - 3b) & -e^2 - e + 5 & x^6 - x^5 + 3x^3 - 4x + 8\\
2 & b - 3 & e^2 - 4 & x^6 - x^5 + 3x^3 - 4x + 8\\
2 & \frac{1}{2}(-b^2 - 3b + 2) & e & x^6 - x^5 + 3x^3 - 4x + 8\\
11 & -2b + 1 & e^2 + e - 9 & x^6 + 13x^5 + 83x^4 + 335x^3 + 913x^2 + 1573x + 1331\\
11 & -b^2 - b + 11 & -e - 4 & x^6 + 13x^5 + 83x^4 + 335x^3 + 913x^2 + 1573x + 1331\\
11 & b^2 + 3b - 7 & -e^2 & x^6 + 13x^5 + 83x^4 + 335x^3 + 913x^2 + 1573x + 1331\\
27 & 3 &0 & x^6 + 81x^4 + 2187x^2 + 19683\\
41 & -b^2 - b + 15 & -2e^2 - 2e + 9 & x^6 + x^5 + 98x^4 + 113x^3 + 4018x^2 + 1681x + 68921\\
41 & b^2 + 3b - 3 & 2e - 1 & x^6 + x^5 + 98x^4 + 113x^3 + 4018x^2 + 1681x + 68921\\
41 & 2b - 5 & 2e^2 - 9 & x^6 + x^5 + 98x^4 + 113x^3 + 4018x^2 + 1681x + 68921\\
\bottomrule
\end{tabular} }
\end{table}

\section{\bf Application to the Gross conjecture for non-solvable Galois number fields}\label{sec:ex-nf-gross}

In \cite{demb09} (see also \cite{ser09}), the second author proved the following theorem:

\begin{thm}\label{thm:gross-conj1} There exists a Galois number field $M$ ramified at $2$ only, with Galois group $\Gal(M/\Q) = \SL_2(\F_{2^8})^2 \rtimes \Z/2\Z$. 
\end{thm}
Theorem~\ref{thm:gross-conj1} settled a conjecture of Gross on the existence of non-solvable Galois number fields which ramify at one prime only for $p=2$. 
The number field $M$ was constructed using Galois representations arising from Hilbert modular forms. In this section, we will explain the very interesting connection between 
Theorem~\ref{thm:gross-conj1} and Conjecture~\ref{conj:gross}. 
 
Let $F = \Q(\zeta_{32})^+ = \Q(\beta)$, where $\beta:=\zeta_{32}+\zeta_{32}^{-1}$ has minimal polynomial $x^8 - 8x^6 + 20x^4 - 16x^2 + 2$.
The space $S_2(1)$, of Hilbert newforms of level $(1)$ and weight $2$ over $F$, has dimension $57$. The Hecke constituents
have dimensions $1, 2, 2, 4, 16$ and $32$. Let $f$ be one of the newforms in the Hecke constituent of dimension $16$; and let 
$L_f = \Q(\{a_\gp(f)\})$ be its coefficient field. Then $L_f$ is a cyclic extension of $K= \Q(\sqrt{5})$. More precisely, $L_f$ is given by the polynomial 
$x^8 + (-10w + 1)x^7 + (26w - 98)x^6 + (1137w - 257)x^5 + (-2048w + 1263)x^4 + (-32915w + 13419)x^3 + (16985w + 55780)x^2 
+ (154252w - 20245) - 170835w - 222209$, where $w = \frac{1 + \sqrt{5}}{2}$. It is a ray class field of conductor $\gc = (48w - 569)$,
which is the product of the primes $(-2w + 1), (w - 10)$, and $(23w - 11)$ above $5$, $89$ and $661$ respectively.  

\begin{lem}\label{lem:newforms} Let $\Gal(F/\Q) = \langle \sigma \rangle$; then there exists $\tau \in \Gal(L_f/K)$ such that
$$a_{\sigma(\gp)}(f) = \tau(a_{\gp}(f))\,\,\text{for all primes}\,\,\gp.$$
\end{lem}

\begin{proof} The prime $31$ splits completely in $F$. There are 8 distinct Hecke eigenvalues $a_{\gp}(f)$ for $\gp \mid 31$. So $f$ cannot be a base
change from any subfield of $F$. (In other words, the stabiliser of $f$ under the action of $\Gal(F/\Q)$ is trivial.) 
Furthermore, since $f$ is in the unique constituent of dimension $16$, its Hecke orbit must be preserved by the 
action of $\Gal(F/\Q)$. Therefore, letting $\Gal(F/\Q) = \langle \sigma \rangle$, we see that there exists $\tau \in \Gal(L_f/K)$ such that
$$a_{\sigma(\gp)}(f) = \tau(a_{\gp}(f))\,\,\text{for all primes}\,\,\gp.$$
\end{proof}

By Theorem~\ref{thm:galois-reps-for-hmf}, we can associate a compatible system of $\lambda$-adic Galois representations
$$\rho_{f, \lambda}:\, \Gal(\Qbar/F) \to \GL_2(L_{f,\lambda}),$$
where $\lambda$ runs over all the primes of $L_f$. The field $M$ in Theorem~\ref{thm:gross-conj1} was constructed from
the reductions of $\rho_{f,\lambda_1}$ and $\rho_{f,\lambda_2}$ modulo the primes $\lambda_1,\lambda_2$ above $2$. 
If we assume that the Eichler-Shimura conjecture holds for $F$, then there is an abelian variety $A_f/F$ of dimension $16$ and conductor $(1)$ such that:
\begin{enumerate}[1.]
\item  $\End_F(A_f)\otimes \Q \simeq L_f$;
\item  For all primes $\lambda$ of $K_f$, $\rho_{A_f, \lambda} = \rho_{f, \lambda}$, where
$\rho_{A_f, \lambda}:\, \Gal(\Qbar/F) \to \GL_2(L_{f,\lambda})$ is the Galois representation in
the $\lambda$-adic Tate module of $A_f$. 
\end{enumerate}
 This would mean that the number field $M$ arises from the field of $2$-torsion of the abelian variety $A_f$ (and its conjugates). 
 
\begin{thm}\label{thm:gross-ex} Assume that the Eichler-Shimura conjecture holds for $F$, and let $A_f$ be the abelian variety attached to $f$. 
Then, there exists an abelian variety $B/\Q$ of dimension $16$ such that the followings hold:
\begin{enumerate}[(i)]
\item The base change $B \times_{\Q}F$ is isogenous to $A_f$;
\item $\End_\Q(B)=\Z[\frac{1+\sqrt{5}}{2}]$;
\item The automorphic representation $\pi_f$ lifts to an automorphic representation $\Pi$ on $\GSpin_{17}(\A_\Q)$, with coefficients in 
$\Q(\sqrt{5})$, such that 
$$L(B, s) = L(\Pi, s) L(\Pi^\tau, s),$$
where $\Gal(\Q(\sqrt{5})/\Q) = \langle \tau \rangle$;
\item The conductor of $B$ is $d_{F/\Q}^{8} = (2^{31})^8 = 2^{248}$. (So, $B$ only has bad reduction at $2$.)
\end{enumerate}
\end{thm}

\begin{proof} The discussion above, together with Lemma~\ref{lem:newforms}, shows that the newform $f$
and the conjecturally attached abelian variety $A_f$ satisfy the conditions of Theorem~\ref{thm:lifts}. So the isogeny class of $A_f$ descends to $\Q$.
\end{proof}

\begin{rem}\rm The abelian varieties $A_f$ and $B$ were discussed in \cite{dv13}. But the connection with lifts
of Hilbert modular forms wasn't clear to the authors at the time.

\end{rem}

\bibliographystyle{amsalpha}

\providecommand{\bysame}{\leavevmode\hbox to3em{\hrulefill}\thinspace}
\providecommand{\MR}{\relax\ifhmode\unskip\space\fi MR }
\providecommand{\MRhref}[2]{%
  \href{http://www.ams.org/mathscinet-getitem?mr=#1}{#2}
}
\providecommand{\href}[2]{#2}

\end{document}